\documentclass[finial,hidelinks,onefignum,onetabnum]{siamart251216}

\usepackage{lipsum}
\usepackage{amsfonts}
\usepackage{graphicx}
\usepackage{epstopdf}
\usepackage{algorithmic}
\usepackage{mathtools}

\usepackage{threeparttable} 
\usepackage{booktabs}
\usepackage{multirow}
\ifpdf
  \DeclareGraphicsExtensions{.eps,.pdf,.png,.jpg}
\else
  \DeclareGraphicsExtensions{.eps}
\fi

\newsiamremark{remark}{Remark}
\newsiamremark{hypothesis}{Hypothesis}
\crefname{hypothesis}{Hypothesis}{Hypotheses}
\newsiamthm{claim}{Claim}
\newsiamremark{fact}{Fact}
\crefname{fact}{Fact}{Facts}
\newtheorem{pro}{Problem}[section]

\headers{Nonlinear Kaczmarz methods for nonlinear systems}{R. Ding, D. Wang, and J. Zou}

\title{Average block nonlinear Kaczmarz methods with adaptive momentum for nonlinear systems of equations\thanks{Submitted to the editors DATE.
\funding{
The work of D. Wang was partially supported by National Natural Science Foundation of China under grants 12271463. 
The work of J. Zou was substantially supported by Hong Kong RGC General Research Fund 
		(projects 14306623
		and 14308322) and NSFC/Hong Kong RGC Joint Research Scheme (project N\_CUHK465/22).
}}}

\author{
	Renjie Ding\thanks{School of Mathematics and Computational Science,  Xiangtan University, Xiangtan, Hunan 411105, P.R. China. (\email{drjmath@smail.xtu.edu.cn}).}
	\and
	Dongling Wang\thanks{ Hunan Research Center of the Basic Discipline Fundamental Algorithmic Theory and Novel Computational Methods, School of Mathematics and Computational Science,  Xiangtan University, Xiangtan, Hunan 411105, P.R. China.
		(\email{wdymath@xtu.edu.cn}).}
	\and
	Jun Zou\thanks{Department of Mathematics, The Chinese University of Hong Kong Shatin, N.T., Hong Kong.
		(\email{zou@math.cuhk.edu.hk}).}
}

\usepackage{amsopn}

\begin{document}

\maketitle

\begin{abstract}
The Kaczmarz method is widely recognized as an efficient iterative algorithm for solving large-scale linear systems, owing to its simplicity and low memory requirements. However, the development of its nonlinear extensions for solving large-scale nonlinear systems has seen limited progress. In this work, we introduce a new family of momentum-accelerated averaging block nonlinear Kaczmarz methods tailored for large-scale nonlinear systems and ill-posed problems. Our contributions are twofold: (1) We develop an adaptive strategy for selecting step sizes and momentum coefficients, leading to a new average block nonlinear Kaczmarz method with adaptive momentum (ABNKAm). This algorithm achieves high computational efficiency by requiring only minimal inner-product computations per iteration, which significantly reduces both arithmetic complexity and memory usage. (2) We establish rigorous convergence 
of the ABNKAm under mild assumptions, proving that the method converges exponentially to the unique solution nearest to the initial point. Moreover, under suitable conditions, we provide a theoretical justification of acceleration of the proposed ABNKAm with momentum. Extensive numerical experiments demonstrate that ABNKAm outperforms existing nonlinear Kaczmarz variants in terms of both iteration count and computational time, with particularly notable gains in large-scale problems.
\end{abstract}



\section{Introduction}\label{sec:intro}
Consider solving the following system of nonlinear equations:
\begin{equation}\label{eq:eqs}
	F(x)=0,
\end{equation}
where \(F:\mathbb{R}^n\rightarrow\mathbb{R}^m\) is a nonlinear vector-valued function and \(x\in\mathbb{R}^n\) is an unknown vector. 
The system of nonlinear equations \(F\) can be rewritten in component form as \(F_i\) for \(i = 1,\cdots,m\), where \(F_i:\mathbb{R}^n\rightarrow\mathbb{R}\) denotes the $i$-th component of \(F\). 

Solving large-scale systems of the form \eqref{eq:eqs} arises widely and significantly in applications such as electrical impedance tomography, circuit design, nonlinear programming \cite{kaltenbacher2008iterative}, systems of nonlinear differential equations \cite{ortega2000iterative}, and machine learning \cite{chen2019homotopy, kawaguchi2016deep,liu2022loss}. Therefore, the efficient solution of the large-scale system of nonlinear equations \eqref{eq:eqs} is of utmost importance. Classical iterative methods, such as Newton-type methods, are widely used but often entail high computational costs or become inapplicable for large-scale or ill-conditioned problems. Thus, there is a pressing need for efficient, low-cost iterative methods. 
For the solution of large-scale linear systems, the Kaczmarz algorithm \cite{kaczmarz1937, strohmer2009randomized}, as a low cost and efficient algorithm, has been remarkably developed. When $F=Ax-b$, where $A\in\mathbb{R}^{m\times n}$ and $b\in\mathbb{R}^m$, the Kaczmarz method orthogonally projects the current iteration point $x_k$ onto the constraint set $\{x|A_{i_k}x=b_{i_k}\}$,  namely it updates the iteration as  
\begin{equation}\label{eq:Kacz}
	x_{k+1}=x_k-\frac{A_{i_k}x_k-b_{i_k}}{\|A_{i_k}\|_2^2}(A_{i_k})^T,
\end{equation}
where the index $i_k$ is selected via random or cyclic sampling from $[m]:=\{1,\dots,m\}$.
This approach significantly reduces the computational cost and storage requirements per iteration. 

The Kaczmarz method was extended for solving large-scale nonlinear systems \cite{wang2022nonlinear}, 
with a family of nonlinear Kaczmarz algorithms. These include the nonlinear Kaczmarz (NK), the nonlinear randomized Kaczmarz (NRK), and the nonlinear uniformly randomized Kaczmarz (NURK) methods. Notably, they established a connection between the NURK method and the stochastic gradient descent (SGD) method \cite{jin2020convergence, robbins1951stochastic}. These methods are based on the following iterative scheme:
\begin{align}\label{eq:nk}
	x_{k+1} = x_k - \frac{F_{i_k}(x_k)}{\|\nabla F_{i_k}(x_k)\|^2_2} \nabla F_{i_k}(x_k).
\end{align}
Here $\nabla F_{i_k}(x)$ denotes the gradient function of $F_{i_k}(x)$. However, selecting index sets from $[m]$ via random or cyclic order strategies may be inefficient and may lead to the selection of indices corresponding to small residual values. To address this limitation, several approximate greedy strategies have been proposed \cite{liu2025greedy, necoara2019faster, zhang2023maximum, zhang2024greedy}, which first construct a controlled index subset $\mathcal{J}_k$ and then select an index from $\mathcal{J}_k$ based on given probability.
To further accelerate the convergence of the NK method, a relaxation parameter $\alpha$ in the iterative formula \eqref{eq:nk} is introduced in \cite{liu2022nonlinear}  to enhance convergence. Specifically, the proposed iteration scheme is formulated as:
\begin{align}\label{eq:rngrk}
	x_{k+1} = x_k - \alpha\frac{F_{i_k}(x_k)}{\|\nabla F_{i_k}(x_k)\|^2_2} \nabla F_{i_k}(x_k).
\end{align}
The Nonlinear Greedy Randomized Kaczmarz method with momentum (NGRKm), integrating a Polyak heavy ball momentum \cite{poljak1964some} into the NGRK framework, was further developed in \cite{liu2025greedy}: 
\begin{equation}\label{eq:ngrkm}
	x_{k+1} = x_k - \alpha \frac{F_{i_k}(x_k)}{\|\nabla F_{i_k}(x_k)\|^2_2} \nabla F_{i_k}(x_k) + \beta (x_k - x_{k-1}),
\end{equation}
where $\alpha$ and $\beta$ are the step size and momentum parameter, respectively. 
Numerical experiments demonstrated the accelerance of NGRKm over NGRK \cite{liu2025greedy}.

Although the above methods have shown improvements over the classical nonlinear Kaczmarz method,
they suffer from two significant drawbacks:
(i) The selection of relaxation parameters $\alpha$ and momentum parameters $\beta$ lacks theoretical guidance.
(ii) These methods update the solution in each iteration by exploiting the knowledge from 
just a single equation. Compared with the block method to be introduced below, which updates 
the approximation by leveraging information from a set of equations at each iteration, this inherent limitation causes their convergence rate to deteriorate substantially as the problem scale increases. This will be 
demonstrated clearly through several typical examples numerically in Section \ref{sec:ne}.

With respect to the drawback (i), 
there are many interesting and in-depth studies \cite{alderman2024randomized, loizou2020momentum, morshed2022sampling, steinerberger2021randomized, zeng2024adaptive, derezinski2025randomized} 
for the linear Kaczmarz method for linear systems.
In particular, an adaptive stochastic heavy-ball momentum method was developed in \cite{zeng2024adaptive}, 
with a significant feature that the choice of the momentum parameter is independent 
of the singular values of the coefficient matrix. However, 
to the best of our knowledge, there is still no similar efficient algorithm available 
for the nonlinear Kaczmarz method.

Building on the success of the linear block Kaczmarz method \cite{gower2015randomized, briskman2015block, necoara2019faster, needell2014paved,  richtarik2020stochastic}, which has demonstrated superiority in handling large-scale linear systems, some studies have extended the concept to the nonlinear equations and developed a nonlinear block Kaczmarz method \cite{tan2024nonlinear, zhang2023maximum, zhang2024greedy, ye2024residual}.
This approach utilizes multiple equations simultaneously in each iteration, with the iterative formula defined as
\begin{align}\label{eq:mrbnk}
	x_{k+1}=x_k-(F_{\mathcal{J}_k}^{\prime}(x_k))^\dagger F_{\mathcal{J}_k}(x_k),
\end{align}
where $F_{\mathcal{J}_k}^{\prime}(x_k)^\dagger$ denotes the Moore-Penrose pseudo-inverse of the submatrix of $F^\prime(x_k)$ and the block index set $\mathcal{J}_k\subset[m]$. Notably, the maximum-residual block nonlinear Kaczmarz method (MRBNK) can also be regarded as an iterative sketch-and-project method applied to the sketched Newton system \cite{yuan2022sketched, zhang2023maximum}. However, these methods require the computing of 
the Moore-Penrose pseudoinverse of the sub-matrix of the Jacobian matrix in each iteration, which increases the computational cost of the methods. To avoid computing the pseudoinverse, a variation of the block Kaczmarz method was proposed 
by first projecting the current iterate onto each row of the selected block matrix, then computing the next iterate 
as a convex combination of these projections:
\begin{align}\label{eq:abnk}
	x_{k+1} = x_k - \alpha_k \left( \sum_{i \in \mathcal{J}_k} \omega_k^{(i)}\frac{F_i(x_k)}{\|\nabla F_i(x_k)\|^2} \nabla F_i(x_k)\right), k \geq 0,
\end{align}
where the weights $\omega_k^{(i)} \in [0,1]$ and satisfy $ \sum_{i\in \mathcal{J}_k}\omega_k^{(i)}=1$ and $\alpha_k\in (0,2)$. 
The method constructed based on iterative formula \eqref{eq:abnk} is named the average block nonlinear Kaczmarz method (ABNK) \cite{xiao2024averaging}. Notably, while the averaging technique in block iteration formula \eqref{eq:abnk} achieves lower computational cost compared to the pseudo-inverse based block scheme \eqref{eq:mrbnk}, its precision may be compromised. This observation motivates the necessity to enhance the convergence rate of averaging techniques without sacrificing their computational efficiency.

In this work, we propose a novel class of averaged block nonlinear Kaczmarz methods with momentum to design more efficient nonlinear Kaczmarz-type algorithms. From a geometric perspective, the selection of step size and momentum parameters is investigated, enabling adaptive update of these parameters using only information from the current iterate. The main contributions are:  

\begin{itemize}
	\item 
  We introduce a new framework that integrates averaging block techniques with momentum acceleration for solving nonlinear systems. This framework is highly flexible, with some important existing methods like ABNK \cite{xiao2024averaging} and NGRKm \cite{liu2025greedy} as the special cases. A key component is its adaptive scheme for selecting the step size and momentum parameters, which overcomes the limitation of manual tuning in prior work and enhances the convergence rate of the basic averaging block method from \cite{xiao2024averaging}. The proposed methods also exhibit superior scalability compared to existing nonlinear Kaczmarz approaches.
	\item
	The proposed method combines the benefits of momentum techniques with averaging block methods to achieve faster convergence while maintaining computational efficiency. Convergence guarantees are provided for the algorithm, and numerical experiments demonstrate practical effectiveness and superior performance compared to existing nonlinear Kaczmarz approaches. 
\end{itemize}

The paper is structured as follows: Section \ref{sec:ABNKM} presents the average block nonlinear Kaczmarz method with  momentum, including its general framework and adaptive-parameter variants.
Section \ref{sec:convergence} provides a rigorous convergence analysis for the proposed method with convergence rate. 
In Section \ref{sec:ne}, numerical experiments are conducted to verify the effectiveness of the proposed algorithms, with 
a comprehensive comparison among the proposed methods and existing nonlinear Kaczmarz methods. 
Section \ref{sec:con} presents some concluding remarks.

\section{The average block nonlinear Kaczmarz method with momentum}\label{sec:ABNKM}
\label{sec:main}

In this section, we propose a new average block nonlinear Kaczmarz method incorporating momentum, and 
an adaptive selection for both the step size and momentum parameters, leading to an adaptive momentum variant 
of the average block nonlinear Kaczmarz method.

\subsection{General framework}\label{sec:general}

To reduce the per-iteration computational cost in \eqref{eq:mrbnk}, we employ a block-averaged nonlinear Kaczmarz strategy. Furthermore, to accelerate convergence, we introduce a momentum term into the iterative scheme \eqref{eq:abnk}, leading to 
the new method of the general form:
\begin{align}\label{eq:abnkm0}
	x_{k+1} = x_k - \alpha_k \left( \sum_{i \in \mathcal{J}_k} \omega_k^{(i)}\frac{F_i(x_k)}{\|\nabla F_i(x_k)\|^2} \nabla F_i(x_k) \right)+\beta_k(x_k-x_{k-1}),  \quad k \geq 0,
\end{align}
where $\alpha_k$ is the step size, $\beta_k$ is the momentum parameter, $\omega_k^{(i)} \in [0,1]$ are weights satisfying $\sum_{i\in \mathcal{J}_k}\omega_k^{(i)} = 1$, and $\mathcal{J}_k$ is the index set selected at the $k$-th iteration. To prioritize the reduction of larger-magnitude components in the residual vector $r_{k} = -F(x_{k})$, we adopt a greedy criterion for selecting $\mathcal{J}_{k}$, following the approach in \cite{bai2018relaxed, ding2025adaptive, morshed2022sampling,zhang2023maximum, xiao2024averaging, tan2024nonlinear}:
\begin{align}\label{eq:greed}
	\mathcal{J}_{k}=\left\{ i_k \left| |F_{i_k}(x_k)|^2\geq \theta \max_{1\leq i\leq m} |F_{i}(x_k)|_2^2 \right.\right\},
\end{align}
where $\theta\in(0,1]$ controls the selectivity of the index set. 
This defines the general framework of the average block nonlinear Kaczmarz method with momentum (ABNKm).

\begin{remark}\label{Remark1}
	The cardinality of the index set $\mathcal{J}_k$ evolves dynamically with the iteration, and it is governed by the control parameter $\theta$. 
	It is crucial to design an algorithm that is robust with respect to parameter $\theta$.
\end{remark}

\begin{remark}\label{Remark2}
	The resulting ABNKm gives rise to some popular existing algorithms, with appropriate selections of the step size $\alpha_k$, 
	the momentum parameter $\beta_k$ and the weight $\omega_k$.  
	The simplest form of the algorithm may use constant step sizes and momentum parameters. 
	In particular, if we select $\beta_k=0$, the proposed ABNKm reduces to the ABNK scheme. 
	More appropriate choices for these key parameters, 
	especially those that can improve the convergence, will be discussed in the subsequent section.
\end{remark}

\subsection{The average block nonlinear Kaczmarz method with adaptive momentum}\label{sec:kernal}
In this section, we propose some effective adaptive strategies for selecting both the step size and momentum parameters, leveraging these mechanisms to accelerate the convergence of the newly proposed ABNKm method \eqref{eq:abnkm0}. 
The resulting algorithm, which explicitly incorporates these adaptive strategies, 
is designated as ABNKAm.

We now address the selection of the step size parameter $\alpha_k$ and momentum parameter $\beta_k$ in 
\eqref{eq:abnkm0}. Our objective is to determine  $\alpha_k$ and $\beta_k$ such that the error $x_{k+1}-x_*$ 
is minimized as follows: 
\begin{equation}
	\begin{aligned}\label{eq:optimal goal}
		\min_{\alpha,\,\beta\in\mathbb{R}} & \quad\|x-x_*\|_2^2 \\
		\text{subject to} & \quad  x = x_k - \alpha d_k +\beta(x_k-x_{k-1}),
	\end{aligned}
\end{equation}
where $d_k:=\sum_{i \in \mathcal{J}_k} \omega_k^{(i)}\frac{F_i(x_k)}{\|\nabla F_i(x_k)\|^2}\nabla F_i(x_k)$. 
If the condition
\begin{align}\label{eq:cond}
	\Delta_k:=\|d_k\|_2^2\|x_k-x_{k-1}\|_2^2-\left\langle d_k,x_k-x_{k-1}\right\rangle^2\neq0
\end{align}
holds, then the solution to the optimization problem \eqref{eq:optimal goal} is explicitly given by
\begin{equation}\label{eq:origin sol}
	\begin{split}
		\left\{
		\begin{aligned}
			\alpha_k&=\frac{\|x_{k}-x_{k-1}\|_{2}^{2} \langle d_{k},x_{k}-x_* \rangle - \langle d_{k},x_{k}-x_{k-1}\rangle \langle x_k-x_{k-1},x_{k}-x_*\rangle} {\Delta_k}, \\
			\beta_k&=\frac{\langle x_{k}-x_{k-1},d_{k}\rangle \langle x_{k}-x_*,d_k \rangle -||d_k||_{2}^{2} \langle x_k-x_{k-1},x_{k}-x_*\rangle}{\Delta_k}.
		\end{aligned}
		\right.
	\end{split}
\end{equation}
Unfortunately, the update formulae \eqref{eq:origin sol} require the knowledge of the true solution $x_*$ 
to equation \eqref{eq:eqs}, and this is clearly impractical. 
We next discuss how to handle the two terms that involves $x_*$, i.e., 
$\langle x_k-x_{k-1}, x_k-x_* \rangle$ and $\langle d_k,x_k-x_{*} \rangle$ for all $k\geq 1$.

On the one hand, from a geometric perspective of the minimization \eqref{eq:optimal goal} we can view 
$x_{k}$ for $k\geq 2$ as the orthogonal projection of $x_*$ onto the affine set
\begin{align}\label{eq:sub space}
	\Pi_{k-1}:=x_{k-1}+\text{Span}\{d_{k-1},x_{k-1}-x_{k-2}\}\,,
\end{align}
hence we have $\langle x_k-x_{k-1}, x_k-x_* \rangle=0$ for  $k\geq 2$.
While for $k=1$, we obtain $x_1$ from the optimization:
\begin{equation}\label{eq:simple approx sol}
	\begin{aligned}
		\min_{\alpha\in\mathbb{R}} & \quad\|x-x_*\|_2^2 \\
		\text{subject to} & \quad  x = x_0 - \alpha d_0.
	\end{aligned}
\end{equation}
This implies that the orthogonality condition $\langle x_k-x_{k-1}, x_k-x_*\rangle=0$ holds for all $k\geq 1$.

On the other hand, in order to deal with $\langle d_k,x_k-x_{*} \rangle$ in \eqref{eq:origin sol}, 
we may linearize the nonlinear operator $F$ at the current iterate $x_k$ via Taylor expansion, 
$
F(x) \approx F(x_k) + F^{\prime}(x_k)(x - x_k).
$
Then we have the approximation
\begin{align}
	\langle d_k,x_k-x_{*} \rangle & =\left\langle \sum_{i \in \mathcal{J}_{k}} \omega_k^{(i)}\frac{F_i(x_k)}{\|\nabla F_i(x_k)\|^2}\nabla F_i(x_k),x_{k}-x_{*} \right\rangle \notag\\
	& =\sum_{i \in \mathcal{J}_{k}} \omega_k^{(i)}\frac{F_i(x_k)}{\|\nabla F_i(x_k)\|^2} \langle \nabla F_i(x_k),x_{k}-x_{*}\rangle \notag\\
	& \approx\sum_{i\in \mathcal{J}_{k}}\omega_{k}^{(i)}\frac{F_{i}(x_{k})}{||\nabla F_{i}(x_{k})||_{2}^{2}}F_{i}(x_{k}) \notag\\
	& =\sum_{i\in \mathcal{J}_{k}}\omega_{k}^{(i)}\frac{F_{i}(x_{k})^{2}}{||\nabla F_i(x_k)||_{2}^{2}}
	\,.\label{eq:Reduced inner product}
\end{align}
Using the above deductions, we can rewrite \eqref{eq:origin sol} approximately as 
\begin{equation}\label{eq: actual sol}
	\begin{split}
		\left\{
		\begin{aligned}
			\alpha_k=\frac{\|x_{k}-x_{k-1}\|_{2}^{2} \cdot \sum_{i\in \mathcal{J}_{k}}\omega_{k}^{(i)}\frac{F_{i}(x_{k})^{2}}{\|\nabla F_{i}(x_{k})\|_{2}^{2}} } {\Delta_k}, \\
			\beta_k=\frac{\langle x_{k}-x_{k-1},d_{k}\rangle \cdot \sum_{i\in \mathcal{J}_{k}}\omega_{k}^{(i)}\frac{F_{i}(x_{k})^{2}}{\|\nabla F_{i}(x_{k})\|_{2}^{2}}}{\Delta_k}. 
		\end{aligned}
		\right.
	\end{split}
\end{equation}

Under the condition that $\Delta_k \neq 0$ as defined in \eqref{eq:cond}, or equivalently, $\text{dim}(\Pi_k) = 2$ for all $k \geq 1$,  the ABNKAm method is well-defined and can be formulated as Algorithm \ref{alg:ABNKAm}.

\begin{algorithm}
	\caption{The average block nonlinear Kaczmarz method with adaptive momentum}
	\label{alg:ABNKAm}
	\begin{algorithmic}[1]
		\STATE \textbf{Input:} The initial guess $x_0$, $F$, parameter $\theta\in(0,1]$, $\tau_a, \tau_r$
		\STATE Choose $k=0$, $r_0=-F(x_0)$ and weights sequence $(\omega_k)_{k\geq0}$
		\WHILE {$\|r_k\|_2\geq \tau_a + \tau_r \|r_0\|_2$}
		\STATE Determine the index set $\mathcal{J}_k=\left\{ i_k||r_k^{(i_k)}|^2\geq\theta\max_{1\leq i\leq m}|r_k^{(i)}|^2 \right\}$
		\STATE Compute the parameters $\alpha_{k}$ and $\beta_{k}$ in \eqref{eq: actual sol}.
		\STATE Update  $x_{k} \leftarrow x_k - \alpha_k \left( \sum_{i \in \mathcal{J}_k} \omega_k^{(i)}\frac{F_i(x_k)}{\|\nabla F_i(x_k)\|^2} \nabla F_i(x_k) \right)+\beta_k(x_k-x_{k-1})$
		\STATE Compute residual $r_k=-F(x_k)$
		\ENDWHILE
		\STATE \textbf{Output:} $x_{k}$
	\end{algorithmic}
\end{algorithm}

In general, the condition $\Delta_k \neq 0$ may not be guaranteed for all $k \geq 1$. 
Even if this is true, it may occur that $\Delta_k$ is relatively small in magnitude for some $k$. 
To enhance the robustness of the algorithm and improve numerical stability, we introduce a modified update rule 
of Algorithm \ref{alg:ABNKAm}.

Specifically, suppose that at the $\ell$-th iteration, either $|\Delta_\ell| < \varepsilon$ for a given tolerance $\varepsilon$, or the condition
\begin{align}\label{beta tru}
	\beta_k \notin (0, \beta_{\max})
\end{align}
holds, where $\beta_{\max}$ is an appropriately selected positive constant. In such cases, the update rule \eqref{eq:optimal goal} is replaced by the solution to the optimization problem:
\begin{equation}\label{eq:cond sol}
	\begin{aligned}
		\min_{\alpha \in \mathbb{R}}  \quad \|x - x_\ell^*\|_2^2 
		\text{   ~~subject to }  \quad x = x_\ell - \alpha d_\ell,
	\end{aligned}
\end{equation}
where $x_\ell^*$ denotes the solution of the linearized model at $x_\ell$, i.e., $F(x_\ell) + F'(x_\ell)(x_\ell^* - x_\ell) = 0$. 
The minimization \eqref{eq:cond sol} shares the same structure as \eqref{eq:simple approx sol}, and 
its minimizer is given by
\begin{equation}\label{eq cond sol sol}
	\alpha_\ell = \frac{\langle d_\ell, x_\ell - x_\ell^* \rangle}{|d_\ell|^2}
	= \frac{\sum_{i \in \mathcal{J}_{\ell}} \omega_{\ell}^{(i)} \frac{F_i(x_\ell)^2}{\|\nabla F_i(x_\ell)\|_2^2}}{|d_\ell|^2}.
\end{equation}

Based on the preceding analysis, we introduce a practical implementation of the average block nonlinear Kaczmarz method with adaptive momentum. For the sake of convenience, we still refer to this implementation as ABNKAm, 
which is summarized in Algorithm \ref{alg:ABNKAm2}.
\begin{algorithm}
	\caption{A practical variant of the average block nonlinear Kaczmarz method with adaptive momentum}
	\label{alg:ABNKAm2}
	\begin{algorithmic}[1]
		\STATE \textbf{Input:} The initial guess $x_0$, $F$, parameters $\theta\in(0,1]$, $\varepsilon$, $\tau_a, \tau_r$
		\STATE Choose $k=0$, $r_0=-F(x_0)$ and weights sequence $(\omega_k)_{k\geq0}$
		\WHILE {$\|r_k\|_2\geq \tau_a + \tau_r \|r_0\|_2$}
		\STATE Determine the index set $\mathcal{J}_k=\left\{ i_k||r_k^{(i_k)}|^2\geq\theta\max_{1\leq i\leq m}|r_k^{(i_k)}|^2 \right\}$
		\STATE Compute the condition $\Delta_k$ in \eqref{eq:cond}
		\STATE  Compute the parameters $\beta_{k}$ in \eqref{eq: actual sol}.
		\IF{$| \Delta_k |\geq\varepsilon$ or $\beta_k \in (0, \beta_{\max})$}
		\STATE Compute the parameters $\alpha_{k}$ in \eqref{eq: actual sol}.
		\ELSE
		\STATE Compute the parameters $\alpha_{k}$ in \eqref{eq cond sol sol}, and let $\beta_{k}=0$ 
		\ENDIF
		\STATE Update  $x_{k} \leftarrow x_k - \alpha_k \left( \sum_{i \in \mathcal{J}_k} \omega_k^{(i)}\frac{F_i(x_k)}{\|\nabla F_i(x_k)\|^2} \nabla F_i(x_k) \right)+\beta_k(x_k-x_{k-1})$
		\STATE Compute residual $r_k=-F(x_k)$
		\ENDWHILE
		\STATE \textbf{Output:} $x_{k}$
	\end{algorithmic}
\end{algorithm}

For the proposed Algorithms \ref{alg:ABNKAm}-\ref{alg:ABNKAm2}, a natural choice for the weight sequence \cite{xiao2024averaging, necoara2019faster} is given by
\begin{align}\label{eq:weights}
	\omega_k^{(i)}=\frac{\|\nabla F_i(x_k)\|^2}{\sum_{i\in \mathcal{J}_k}\|\nabla F_i(x_k)\|^2} \quad\text{for all } k\geq 0\,.
\end{align}
Building upon the preceding analysis, the iterative formula \eqref{eq:abnkm0} admits the following three cases:

	(1) The sequences for both the step sizes and the momentum parameters 
	are kept to be constant (i.e., $\alpha_k = \alpha$, $\beta_k = \beta$ for all $k$), then the iterative formula is simplified to the form
	\begin{align}\label{eq:abnkmc}
		x_{k+1}=x_k-\alpha \frac{(F^{\prime}_{\mathcal{J}_k}(x_k))^TF_{\mathcal{J}_k}(x_k)}{\|F^{\prime}_{\mathcal{J}_k}(x_k)\|^2}+\beta(x_k-x_{k-1}).
	\end{align}
	
	(2) The sequences for both the step sizes and the momentum parameters are dynamically adjusted 
	via formula \eqref{eq: actual sol}, then the iterative formula reads as 
	\begin{align}\label{eq:abnkam1}
		x_{k+1}=x_k-\alpha_k \frac{(F^{\prime}_{\mathcal{J}_k}(x_k))^TF_{\mathcal{J}_k}(x_k)}{\|F^{\prime}_{\mathcal{J}_k}(x_k)\|^2}+\beta_k(x_k-x_{k-1}).
	\end{align}
	with
	\begin{align*}
		\alpha_k&=\frac{\| x_k-x_{k-1}\|
			_2^2 \|F_{\mathcal{J}_k}(x_k)\|_2^2 \|F^{\prime}_{\mathcal{J}_k}(x_k)\|^2}{\Delta_k},\notag \\
		\beta_k&=\frac{\langle x_k-x_{k-1}, (F^{\prime}_{\mathcal{J}_k}(x_k))^TF_{\mathcal{J}_k}(x_k)\rangle \|F_{\mathcal{J}_k}(x_k)\|^2_2}{\Delta_k}.\notag \\
		\Delta_k&=\|(F^{\prime}_{\mathcal{J}_k}(x_k))^TF_{\mathcal{J}_k}(x_k)\|^2_2 \cdot \|x_k-x_{k-1}\|^2_2- \langle(F^{\prime}_{\mathcal{J}_k}(x_k))^TF_{\mathcal{J}_k}(x_k),x_k-x_{k-1}\rangle^2, \notag
	\end{align*}

	(3) A more practical implementation of ABNKAm adopts the following hybrid iteration scheme:
\begin{equation}\label{eq:abnkam2}
x_{k+1} = 
\begin{cases}
x_k - \alpha_k \dfrac{(F'_{\mathcal{J}_k}(x_k))^{\mathsf{T}} F_{\mathcal{J}_k}(x_k)}{\|F'_{\mathcal{J}_k}(x_k)\|^2} + \beta_k (x_k - x_{k-1}), \\[2ex]
\quad \text{if } |\Delta_k| \geq \varepsilon \ \text{and} \ \beta_k \in (0, \beta_{\max}); \\[2ex]
x_k - \dfrac{\|F_{\mathcal{J}_k}(x_k)\|_2^2}{\bigl\|(F'_{\mathcal{J}_k}(x_k))^{\mathsf{T}} F_{\mathcal{J}_k}(x_k)\bigr\|_2^2} (F'_{\mathcal{J}_k}(x_k))^{\mathsf{T}} F_{\mathcal{J}_k}(x_k), 
\quad \text{otherwise}.
\end{cases}
\end{equation}
	Here $\Delta_k$, $\alpha_k$ and $\beta_k$ are the same as in \eqref{eq:abnkam1}. This hybrid approach effectively balances numerical stability with computational efficiency, and ABNKAm takes this form in all subsequent numerical experiments. 

We conclude this subsection with some important remarks about ABNKAm.
\begin{remark}
	The ABNKAm method based on the iterative scheme \eqref{eq:abnkam2} requires at most $4$ vector inner products 
	at each iteration, while $(F^{\prime}_{\mathcal{J}_k}(x_k))^TF_{\mathcal{J}_k}(x_k)$ can be implemented 
	through matrix-vector products without forming the sub-Jacobian matrix.
\end{remark}

\begin{remark}
	Our proposed framework presents a general computational paradigm that adapts updates rules 
	based on local geometric conditions. It is clear 
	that the popular average block nonlinear Kaczmarz method with extrapolated step sizes (ABNK2) \cite{xiao2024averaging} is 
	as a special case, where momentum is not considered.
\end{remark}

\section{Convergence analysis}\label{sec:convergence}

In this section, we establish the convergence properties of the average block nonlinear Kaczmarz method with adaptive momentum 
in subsection\,\ref{sec:kernal}. We begin by introducing some necessary assumptions and preliminary results 
that are used in our subsequent analysis.

\begin{hypothesis}\label{as:asc}
	Let $D \subseteq \mathbb{R}^n$ be a bounded closed domain, $x_0 \in D$ and $F : D \to \mathbb{R}^m$ a nonlinear vector-valued function. The following conditions are assumed:
	
	\begin{itemize}
		\item[(i)] Each component $F_i : D \to \mathbb{R}$ of $F = (F_1, \dots, F_m)^{\mathsf{T}}$ 
		is differentiable.
		
		\item[(ii)] There exists $\delta_0 > 0$ such that
		$B_{\delta_0}(x_0) \subseteq D$.

		\item[(iii)] For all $i \in \{1, \dots, m\}$ and $x, \widetilde{x} \in B_{\delta_0}(x_0)$, there exists $\eta_i \in [0, \eta)$ with 
		$\eta = \max\limits_{1 \leq i \leq m} \eta_i < {1}/{2}$ such that
		\begin{equation}\label{eq:local cond}
			\left| F_i(x) - F_i(\widetilde{x}) - \nabla F_i(x)^{\mathsf{T}} (x - \widetilde{x}) \right| \leq \eta_i \left| F_i(x) - F_i(\widetilde{x}) \right|.
		\end{equation}

		\item[(iv)] For the constant $\delta_0$ from (ii), it holds that
		\begin{equation} \label{eq:sol and ini}
			B_{\delta_0/4}(x_0) \cap \ker(F) \neq \emptyset\,.
		\end{equation}

		\item[(v)] For the same $\delta_0$, there exists $\widetilde{x} \in B_{\delta_0}(x_0)$ such that
		\begin{equation}
			F'(\widetilde{x})^{\mathsf{T}} e_i \neq 0 \quad \text{for all } i \in {1, \dots, m},
		\end{equation}
		where $e_1, \dots, e_m \in \mathbb{R}^m$ form the standard orthonormal basis of $\mathbb{R}^m$.
	\end{itemize}
\end{hypothesis}

Assumptions \ref{as:asc} are standard in the literature \cite{hanke1995convergence, haltmeier2007kaczmarz, kaltenbacher2008iterative, tan2024nonlinear}.
Condition $(iii)$ is the so-called local tangential cone condition, a natural requirement for Kaczmarz-type iterative methods; detailed explanations and heuristic justifications can be found in \cite[page 12]{kaltenbacher2008iterative}. Moreover, the condition 
$\eta < {1}/{2}$ plays a pivotal role in ensuring convergence of the nonlinear Kaczmarz method, as will become evident in Lemmas \ref{le:case2}. 
Conditions $(ii)$ and $(iv)$ together imply the existence of a unique $x_0$-minimum-norm solution, as will be shown in Lemma \ref{le:min sol}.
Finally, Assumption $(v)$ is necessary for the validity of several of our subsequent key results, 
as established in Lemma \ref{le:as odd}. In fact, condition $(v)$ is less restrictive 
than the usual condition that requires the Jacobian matrix of $F$ to be row-bounded below or to satisfy full column rank conditions.

\begin{lemma}\label{le:cy} \cite{hanke1995convergence, haltmeier2007kaczmarz, kaltenbacher2008iterative, tan2024nonlinear, xiao2024averaging}
	Let $\mathcal{J} \subseteq \{1, 2, \ldots, m\}$. Under Assumption \ref{as:asc}(ii)-(iii), it holds for all $x,\widetilde{x}\in B_{\delta_0}(x_0)$ that 
	\begin{align*}
		&\frac{1}{1+\eta} \|F_{\mathcal{J}}'(x)(x - \widetilde{x})\|_2 \leq \|F_{\mathcal{J}}(x) - F_{\mathcal{J}}(\widetilde{x})\|_2 \leq \frac{1}{1-\eta} \|F_{\mathcal{J}}'(x)(x - \widetilde{x})\|_2, \\
		&\left\|F_{\mathcal{J}}(x) - F_{\mathcal{J}}(\widetilde{x}) - F'_{\mathcal{J}}(x)(x - \widetilde{x})\right\|_{2}^{2} \leq \eta^{2} \left\|F_{\mathcal{J}}(x) - F_{\mathcal{J}}(\widetilde{x})\right\|_{2}^{2}, \\
		&\left\|F_{\mathcal{J}}(x) - F_{\mathcal{J}}(\widetilde{x})\right\|_{2}^{2} \geq \frac{1}{1 + \eta^{2}} \left\|F'_{\mathcal{J}}(x)(x - \widetilde{x})\right\|_{2}^{2}.
	\end{align*}
\end{lemma}

\begin{lemma}\label{le:min sol}\cite{kaltenbacher2008iterative, tan2024nonlinear}
	Let $x_0\in D$ and the vector-valued function $F$ satisfies Assumption \ref{as:asc}(ii)-(iv). Then
	
	(i) For any \( x, \widetilde{x} \in B_{\delta_0}(x_0) \) with \( F(x) = F(\widetilde{x}) \), it holds that \( x - \widetilde{x} \in \ker(F'(x)) = \ker(F'(\widetilde{x})) \).
	
	(ii) There exists a unique \( x_0 \)-minimum-norm solution $x_*\in B_{\delta_0}(x_0)$ to \eqref{eq:eqs}, satisfying 
	\begin{equation}\label{eq:sol x0}
		x_* - x_0 \in \ker(F'(x_*))^\perp.
	\end{equation}
\end{lemma}

The second part of Lemma \ref{le:min sol} establishes that $x_*$ is the unique $x_0$-minimum-norm element in $\ker(F) \cap B_{\delta_0}(x_0)$ and satisfies condition \eqref{eq:sol x0}. Indeed, by Lemma \ref{le:min sol}(i) and \eqref{eq:sol x0}, if there exists another point $x^\dag \in \ker(F) \cap B_{\delta_0}(x_0)$, then
\begin{align}\nonumber
	\|x^{\dag} - x_0\|^2_2 &= \|x^{\dag} - x_*\|^2_2 + \|x_* - x_0\|^2_2 + 2\langle x^{\dag} - x_*, x_* - x_0 \rangle \\
	&= \|x^{\dag} - x_*\|^2_2 + \|x_* - x_0\|^2_2 > \|x_* - x_0\|^2.
\end{align}

Furthermore, in the linear case $F(x) = Ax - b$, condition \eqref{eq:sol x0} implies that $x_*$ corresponds to the best approximation of $x_0$ onto the solution manifold $\{x \mid Ax = b\}$.

\begin{lemma}\label{le:as odd}\cite{tan2024nonlinear}
	Let $x_0 \in D$ and suppose $F$ satisfies Assumptions \ref{as:asc}(ii)--(iii). Then, for every $i \in {1, 2, \dots, m}$, $F_i(x) = 0$ for all $x \in B_{\delta_0}(x_0)$ if and only if there exists $\widetilde{x}_i \in B_{\delta_0}(x_0)$ such that $F'(\widetilde{x}_i)^\top e_i = 0$.
\end{lemma}

Lemma \ref{le:as odd} demonstrates the necessity of condition (v) in Assumption \ref{as:asc}. If condition (v) is not satisfied, then by Lemma \ref{le:as odd}, there would exist some $i_0 \in \{1, 2, \dots, m\}$ such that $F_{i_0}(x) = 0$ for all $x \in B_{\delta_0}(x_0)$.

\begin{lemma}\label{le:mvs} \cite{horn2012matrix}
	Let $A \in \mathbb{R}^{m \times n}$ be any nonzero real matrix. For every vector $u \in \text{range}(A)$, we have 
	$\sigma_{\min}^2(A) \|u\|_2^2 \leq \|A^T u\|_2^2 \leq \sigma_{\max}^2(A) \|u\|_2^2$,
	where $\text{range}(A)$, $\sigma_{\min}(A)$ and $\sigma_{\max}(A)$ are the column space, the nonzero minimum and maximum singular values of $A$, respectively.
\end{lemma}

Before formally analyzing the convergence of the algorithm, we first comment on the computability of the iterative schemes \eqref{eq:abnkam1} and \eqref{eq:abnkam2}.
As discussed in Section \ref{sec:kernal}, the computability of \eqref{eq:abnkam1} relies on 
the condition $\Delta_k \neq 0$. In this case, we will illustrate the acceleration effect of the momentum term and 
make the convergence analysis of Algorithm \ref{alg:ABNKAm} in Section \ref{sec:con 1}. 
Regarding the computability of Algorithm \ref{alg:ABNKAm2}, it is sufficient to ensure that $\|(F^{\prime}_{\mathcal{J}_k}(x_k))^\mathsf{T} F_{\mathcal{J}_k}(x_k)\|^2_2\neq 0$. We have the following more specific conclusions. 
\begin{lemma}\cite{tan2024nonlinear}
	Assume \( x_0 \in D \) and \( F \) satisfies Assumptions \ref{as:asc}. If \( x_k \in B_{\delta_0(x_0)} \) and \( F(x_k) \neq 0 \), 
	then 
	$\|(F^{\prime}_{\mathcal{J}_k}(x_k))^\mathsf{T} F_{\mathcal{J}_k}(x_k)\|^2_2\neq 0$.
\end{lemma}

\begin{corollary}
	Let $x_0 \in D$ and suppose $F$ satisfies Assumptions \ref{as:asc}. If $x_k \in B_{\delta_0}(x_0)$ and $F(x_k) \neq 0$, then the iterate $x_{k + 1}$ defined by \eqref{eq:abnkam2} is computable.
\end{corollary}

\subsection{Convergence analysis of Algorithm \ref{alg:ABNKAm}}\label{sec:con 1}
We first establish the convergence of Algorithm \ref{alg:ABNKAm}. 
A key step in our analysis is to derive a recursive estimate for the error sequence $x_{k}-x_*$, 
which will be shown to satisfy certain contraction conditions. We shall need the following basic 
recurrence estimate.

\begin{lemma}\label{eq: key_le}\cite{loizou2020momentum}
	Fix $E_{1} = E_{0} \geq 0$ and let $\{E_{k}\}_{k \geq 0}$ be a nonnegative real sequence satisfying that
	$E_{k+1} \leq a_1 E_{k} + a_2 E_{k-1}$ for all $k \geq 1$,
	where $a_2 \geq 0$, $a_1 + a_2 < 1$ and at least one of the coefficients $a_1$, $a_2$ is positive. Then the sequence satisfies the estimae $E_{k+1} \leq q^k (1 + \xi) E_0$ for all $k \geq 1$, where $q = \frac{a_1 + \sqrt{a_1^2 + 4a_2}}{2}$ and $\xi = q - a_1 \geq 0$. Moreover, we have
	$ a_1 + a_2\leq q<1$ and the equality holds if and only if $a_2 = 0$ (or $q = a_1$ and $\xi = 0$).
\end{lemma}

Following Lemma \ref{eq: key_le}, we now derive a recursive estimate 
of the form $E_{k+1} \leq a_1 E_k + a_2 E_{k-1}$ associated with Algorithm \ref{alg:ABNKAm}, 
and then achieve the convergence of Algorithm \ref{alg:ABNKAm} (see Theorem \ref{th:ideal}).

\begin{lemma}\label{le:estimate1}
	Suppose that  $F(x)$ and $x_0$ satisfies Assumption \ref{as:asc} and $x_*\in D$ is the unique $x_0$-minimum-norm solution of the nonlinear equation \eqref{eq:eqs}.
	Assume that $x_j\in B_{\delta_0}(x_0)$ for all $j\in\{1,\cdots,k\}$ and $\Delta_k\neq 0$. 
	Then the new approximation $x_{k+1}$  generated by the iteration \eqref{eq:abnkam1} satisfies 
	\vskip -15pt  
	\begin{equation}
		\begin{aligned}\label{eq:est}
			\|x_{k+1}-x_*\|_2^2
			\leq &(1+3\beta_k+2\beta_k^2)\|x_k-x_*\|_2^2 + (2\beta_k^2+\beta_k)\|x_{k-1}-x_*\|^2_2\\
			&-(1-2\eta+\cot^2\theta_k)\frac{\left\|F_{\mathcal{J}_k}(x_k)\right\|_2^2}{\sigma^2_{\max}(F_{\mathcal{J}_k}^{\prime}(x_k))\sin^2\theta_k},
		\end{aligned}
	\end{equation}
	where $\theta_k:=\arccos \left(\frac{F_{\mathcal{J}_k}(x_k)^TF^{\prime}_{\mathcal{J}_k}(x_k) (x_k-x_{k-1})}{\|F^{\prime}_{\mathcal{J}_k}(x_k)^TF_{\mathcal{J}_k}(x_k)\|\cdot \|x_k-x_{k-1}\|} \right)$.
\end{lemma}

\begin{proof}
	It follows from iteration \eqref{eq:abnkam1} that
	\begin{equation}
		\begin{aligned}\nonumber
			\left\|x_{k+1}-x_*\right\|_2^2
			&=\left\|x_k-\alpha_k \frac{(F^{\prime}_{\mathcal{J}_k}(x_k))^TF_{\mathcal{J}_k}(x_k)}{\|F^{\prime}_{\mathcal{J}_k}(x_k))\|^2}+\beta_k(x_k-x_{k-1})-x_* \right\|_2^2\\
			&=\left\|x_k-\alpha_k \frac{(F^{\prime}_{\mathcal{J}_k}(x_k))^TF_{\mathcal{J}_k}(x_k)}{\|F^{\prime}_{\mathcal{J}_k}(x_k))\|^2}-x_* \right\|_2^2+\beta_k^2\left\|x_k-x_{k-1}\right\|_2^2\\&+2\left\langle x_k-\alpha_k\frac{(F_{\mathcal{J}_k}^{\prime}(x_k))^TF_{\mathcal{J}_k}(x_k)}{\left\|F_{\mathcal{J}_k}^{\prime}(x_k)\right\|_2^2}
			-x_*,\beta_k (x_k-x_{k-1})\right\rangle\\
			&:=I_{1}+I_2+I_3.
		\end{aligned}
	\end{equation}
	
	We now estimate $I_{1}$, $I_2$ and $I_3$ one by one. Firstly, we can directly derive 
	\begin{equation}\nonumber
		\begin{aligned}
			I_{1} & =\left\|x_k-x_*\right\|_2^2+\left\|\alpha_k\frac{(F_{\mathcal{J}_k}^{\prime}(x_k))^TF_{\mathcal{J}_k}(x_k)}{\left\|F_{\mathcal{J}_k}^{\prime}(x_k)\right\|_2^2}\right\|_2^2
			-2\left\langle\alpha_k\frac{(F_{\mathcal{J}_k}^{\prime}(x_k))^TF_{\mathcal{J}_k}(x_k)}{\left\|F_{\mathcal{J}_k}^{\prime}(x_k)\right\|_2^2},x_k-x_*\right\rangle \\
			&  =\left\|x_k-x_*\right\|_2^2+\left\|\frac{\|x_k-x_{k-1}\|^2_2\|F_{\mathcal{J}_k}(x_k)\|^2_2}{\Delta_k}(F^{\prime}_{\mathcal{J}_k}(x_k))^TF_{\mathcal{J}_k}(x_k)\right\|_2^2 \\
			&+2\left\langle-\frac{\|x_k-x_{k-1}\|^2_2\|F_{\mathcal{J}_k}(x_k)\|^2_2}{\Delta_k}(F^{\prime}_{\mathcal{J}_k}(x_k))^TF_{\mathcal{J}_k}(x_k),x_k-x_*\right\rangle \\
			& =2\frac{\|x_k-x_{k-1}\|^2_2\|F_{\mathcal{J}_k}(x_k)\|^2_2}{\Delta_k}F_{\mathcal{J}_k}(x_k)^T \left[ F_{\mathcal{J}_k}(x_k)-F_{\mathcal{J}_k}(x_*)-F^{\prime}_{\mathcal{J}_k}(x_k)(x_k-x_*)\right]+ \\
			&\left\|x_k-x_*\right\|_2^2+\frac{\|x_k-x_{k-1}\|^2_2\|F_{\mathcal{J}_k}(x_k)\|^4_2}{\Delta_k} \left[ \frac{\|x_k-x_{k-1}\|^2_2}{\Delta_k}\left\|F^{\prime}_{\mathcal{J}_k}(x_k)^TF_{\mathcal{J}_k}(x_k)\right\|_2^2-2 \right]. 
		\end{aligned}
	\end{equation}
	By Cauchy-Schwarz inequality, Lemma \ref{le:cy}(ii) and Lemma \ref{le:mvs}, we have 
	\begin{equation}\nonumber
		\begin{aligned}
			I_{1} &\leq \left\|x_k-x_*\right\|_2^2+2\eta\frac{\|x_k-x_{k-1}\|^2_2\|F_{\mathcal{J}_k}(x_k)\|^2_2}{\Delta_k}\left\|F_{\mathcal{J}_k}(x_k)\right\|_2^2\\
			&+\frac{\|x_k-x_{k-1}\|^2_2\|F_{\mathcal{J}_k}(x_k)\|^4_2}{\Delta_k} \left[ \frac{\|x_k-x_{k-1}\|^2_2}{\Delta_k}\left\|F^{\prime}_{\mathcal{J}_k}(x_k)^TF_{\mathcal{J}_k}(x_k)\right\|_2^2-2 \right] \\
			&= \left\|x_k-x_*\right\|_2^2-\frac{\|F_{\mathcal{J}_k}(x_k)\|^4_2}{\|F^{\prime}_{\mathcal{J}_k}(x_k)^TF_{\mathcal{J}_k}(x_k)\|^2_2(1-\cos^2\theta_k)} \left[ 2(1-\eta)-\frac{1}{1-\cos^2\theta_k}\right] \\
			&\leq \left\|x_k-x_*\right\|_2^2-\frac{\|F_{\mathcal{J}_k}(x_k)\|^2_2}{\sigma^2_{\max}(F^{\prime}_{\mathcal{J}_k}(x_k))\sin^2\theta_k} \left[2(1-\eta)-\frac{1}{\sin^2\theta_k}\right].
		\end{aligned}
	\end{equation}
	
	For $I_3$, we have
	\begin{equation}\nonumber
		\begin{aligned}
			I_3&=2\left\langle x_k-\alpha_k\frac{(F_{\mathcal{J}_k}^{\prime}(x_k))^TF_{\mathcal{J}_k}(x_k)}{\left\|F_{\mathcal{J}_k}^{\prime}(x_k)\right\|_2^2}-x_*,\beta_k (x_k-x_{k-1})\right\rangle\\
			&=2\beta_k\langle x_k-x_*,x_k-x_{k-1}\rangle - 2\left\langle \alpha_k\frac{(F_{\mathcal{J}_k}^{\prime}(x_k))^TF_{\mathcal{J}_k}(x_k)}{\left\|F_{\mathcal{J}_k}^{\prime}(x_k)\right\|_2^2},\beta_k (x_k-x_{k-1})\right\rangle\\
			&:=I_{3,1}+I_{3,2}.
		\end{aligned}
	\end{equation}
	Using the Cauchy-Schwarz inequality and Lemma \ref{le:mvs}, we can deduce 
	\begin{equation}\nonumber
		\begin{aligned}
			I_{3,2}&=- 2\left\langle \alpha_k\frac{(F_{\mathcal{J}_k}^{\prime}(x_k))^TF_{\mathcal{J}_k}(x_k)}{\left\|F_{\mathcal{J}_k}^{\prime}(x_k)\right\|_2^2},\beta_k (x_k-x_{k-1})\right\rangle\\
			&=-2\frac{\|x_k-x_{k-1}\|^2_2\|F_{\mathcal{J}_k}(x_k)\|^4_2}{\Delta^2_k}\left\langle F^{\prime}_{\mathcal{J}_k}(x_k)^TF_{\mathcal{J}_k}(x_k),x_k-x_{k-1}\right\rangle^2\\
			&=-2\frac{\|x_k-x_{k-1}\|^4_2\|F_{\mathcal{J}_k}(x_k)\|^4_2\|F^{\prime}_{\mathcal{J}_k}(x_k)^TF_{\mathcal{J}_k}(x_k)\|^2_2\cos^2\theta_k}{\|F^{\prime}_{\mathcal{J}_k}(x_k)^TF_{\mathcal{J}_k}(x_k)\|^4_2\|x_k-x_{k-1}\|^4_2(1-\cos^2\theta_k)^2}\\
			&\leq -2\frac{\|F_{\mathcal{J}_k}(x_k)\|^2_2\cos^2\theta_k}{\sigma^2_{\max}(F^{\prime}_{\mathcal{J}_k}(x_k))(1-\cos^2\theta_k)^2}\,,
		\end{aligned}
	\end{equation}
	while we can estimate $I_{2}$ and $I_{3,1}$ together as follows:
	\begin{equation}\nonumber
		\begin{aligned}
			I_{2}+I_{3,1}
			&=\beta_k^2\left\|x_k-x_{k-1}\right\|_2^2+2\beta_k\langle x_k-x_*,x_k-x_{k-1}\rangle \\
			&=\beta_k^2\left\|x_k-x_{k-1}\right\|_2^2+2\beta_k \left\| x_k-x_*\right\|_2^2 +2\beta_k\langle x_k-x_*,x_*-x_{k-1}\rangle \\
			&=\beta_k^2\left\|x_k-x_{k-1}\right\|_2^2+2\beta_k \left\| x_k-x_*\right\|_2^2 +\beta_k\left\|x_k-x_{k-1}\right\|_2^2\\
			&\quad -\beta_k \left\| x_k-x_*\right\|_2^2-\beta_k\left\| x_{k-1}-x_*\right\|_2^2\\
			&\leq(2\beta_k^2+\beta_k)\left\| x_{k-1}-x_*\right\|_2^2+(2\beta_k^2+3\beta_k)\left\| x_{k}-x_*\right\|_2^2
		\end{aligned}
	\end{equation}
	where we have used in the third equality above the identity that 
	$2\langle A-C, C-B\rangle = \|A-B\|_2^2 - \|C-B\|_2^2 - \|A-C\|_2^2$
	for any vectors $A, B,C\in\mathbb{R}^n$. 
	
	By combining all the estimates for $I_1, I_2$ and $I_3$, the expected estimate follows. 
\end{proof}

To further investigate the acceleration mechanism of the momentum effect, we now derive some direct comparison 
between the rate of convergence 
of Algorithm \ref{alg:ABNKAm} (under $\Delta_k \neq 0$ for all $k \geq 0$) with that of the ABNK method \cite{xiao2024averaging} (the scheme \eqref{eq:abnk}). 
For this purpose, we first introduce some constants.

Let $\varepsilon_0=\frac{\eta(1+\eta^2)+\sqrt{\eta^2(1+\eta^2)^2+(1+\eta^2)(1-2\eta^2-4\eta^3)})}{1-2\eta^2-4\eta^3}$, 
where $\eta\in(0,\frac{1}{2})$ is given in Assumption \ref{as:asc}(iii). Then $\varepsilon_0 \in (0,1)$.
For any $\gamma\in (\varepsilon_0/2,  \varepsilon_0)$,
we define 
$C_1(\eta, \gamma):= \sqrt{\frac{(1-2\eta\gamma)(1+\eta^2)}{(1-2\eta^2-4\eta^3)\gamma^2}}$,
and 
\begin{align*}
	C_2(\eta, \gamma) &:= \min \Big\{ \frac{\sqrt{(1+3d^2)^2+8(1+d^2)(d^4+d^2c-d^2)}-(1+3d^2)}{4(1+d^2)}, \frac{\sqrt{1+c}-1}{2} \Big\},
\end{align*}
where $c=\frac{1-2\eta\gamma}{(1+\eta^2)K^2\gamma^2}$ and $d=1-\frac{1-2\eta}{1+\eta^2}$.
We can easily verify that \( C_1(\eta, \gamma) > 1 \) and \( C_2(\eta, \gamma) > 0 \).

\begin{theorem}\label{th:ideal}
	Let $F(x)$ satisfy Assumption \ref{as:asc} with $x_*$ being the unique $x_0$-minimum-norm solution of $F(x)=0$. Assume $\ker(F'(x_*)) \subseteq \ker(F'(x))$ for all $x \in B_{\delta_0/2}(x_*)$, and let $\{x_k\}$ be the sequence generated by Algorithm \ref{alg:ABNKAm} with $\Delta_k \neq 0$ for all $k \geq 0$.
	We assume that there exists $\gamma\in (\varepsilon_0/2,  \varepsilon_0)$ such that the following conditions hold for all \(k\).
	\begin{enumerate}
		\item[(C1)] $\kappa(F'_{\mathcal{J}k}(x_k)) \leq K$ for some $K \in [1, C_1(\eta, \gamma))$,
		\item[(C2)] $\beta_k \in (0, \beta_{\max})$ for some $\beta_{\max} \in (0, C_2(\eta, \gamma))$,
		\item[(C3)] $\sin^2\theta_k \leq \gamma$,
	\end{enumerate}
	then the following conclusions hold:
	
	(i) $\{x_k\} \subseteq B_{\delta_0/2}(x_*) \subseteq B_{\delta_0}(x_0)$,
	
	(ii) If $F(x_k) \neq 0$ for some $k$, then $x_{k+1} - x_* \in \ker(F'(x_*))^\perp$ and
	\[
	\|x_{k+1} - x_*\|_2^2 \leq q^k(1+\xi)\|x_0 - x_*\|_2^2,
	\]
	where $q = \frac{b_1 + \sqrt{b_1^2 + 4b_2}}{2}$, $\xi = q - b_1$ with
	$b_1=1+3\beta_{\max}+2\beta_{max}^2-\frac{1-2\eta\delta}{(1+\eta^2)K^2 \delta^2},\,b_2=2\beta_{\max}^2+\beta_{\max} $.
	
	Furthermore, the estimate $q< d^2 < 1$ holds for the rate $q$ of convergence.
\end{theorem}

\begin{proof}	
	We apply the mathematical induction. 
	From Assumption \ref{as:asc}(iv) and Lemma \ref{le:case2}, we have $x_0\in B_{\frac{\delta_0}{4}}(x_*)\subseteq B_{\delta_0}(x_0),x_1\in B_{\frac{\delta_0}{2}}(x_*)\subseteq B_{\delta_0}(x_0)$ and $x_0-x_*,x_1-x_*\in \ker(F'(x_*))^{\perp}$.
	Accordingly, we may assume without loss of generality that $x_j \in B_{\frac{\delta_0}{2}}(x_*) \subseteq B_{\delta_0}(x_0),\, x_j - x_* \in \ker(F'(x_*))^\perp$ hold for all $j\in \{0,1,\cdots,k\}$. It remains 
	to show that \( x_{k+1} \in B_{\frac{\delta_0}{2}}(x_*) \subseteq B_{\delta_0}(x_0) \) and \( x_{k+1} - x_* \in \ker(F'(x_*))^\perp \).
	It follows from \eqref{eq:est} that 
	\begin{equation}\nonumber
		\begin{aligned}
			\|x_{k+1}-x_*\|_2^2
			&\leq (1+3\beta_k+2\beta_k^2)\|x_k-x_*\|_2^2 + (2\beta_k^2+\beta_k)\|x_{k-1}-x_*\|^2_2\\
			&-(1-2\eta+\cot^2\theta_k)\frac{\left\|F_{\mathcal{J}_k}(x_k)\right\|_2^2}{\sigma^2_{\max}(F_{\mathcal{J}_k}^{\prime}(x_k))\sin^2\theta_k}.
		\end{aligned}
	\end{equation}
	For the last term on the right-hand side, we can rewrite and estimate as follows:
	\begin{equation}\nonumber
		\begin{aligned}
			&-(1-2\eta+\cot^2\theta_k)\frac{\left\|F_{\mathcal{J}_k}(x_k)\right\|_2^2}{\sigma^2_{\max}(F_{\mathcal{J}_k}^{\prime}(x_k))\sin^2\theta_k}\\
			=&-(1-2\eta+\cot^2\theta_k)\frac{\left\|F_{\mathcal{J}_k}(x_k)-F_{\mathcal{J}_k}(x_*)\right\|_2^2}{\sigma^2_{\max}(F_{\mathcal{J}_k}^{\prime}(x_k))\sin^2\theta_k}\\
			\leq& -\frac{(1-2\eta+\cot^2\theta_k)}{(1+\eta^2)\sigma^2_{\max}(F_{\mathcal{J}_k}^{\prime}(x_k))\sin^2\theta_k}\left\|F^{\prime}_{\mathcal{J}_k}(x_k)(x_k-x_*)\right\|^2_2 \\
			\leq& -\frac{(1-2\eta+\cot^2\theta_k)}{(1+\eta^2)\kappa^2(F_{\mathcal{J}_k}^{\prime}(x_k))\sin^2\theta_k}\left\|x_k-x_*\right\|^2_2,
		\end{aligned}
	\end{equation}
	where we have used Lemma \ref{le:cy} in the first inequality and Lemma \ref{le:mvs} in the second respectively.  
	This implies 
	\begin{equation}\nonumber
		\begin{aligned}
			\|x_{k+1}-x_*\|_2^2
			&\leq (1+3\beta_k+2\beta_k^2-\frac{(1-2\eta+\cot^2\theta_k)}{(1+\eta^2)\kappa^2(F_{\mathcal{J}_k}^{\prime}(x_k))\sin^2\theta_k})\|x_k-x_*\|_2^2\\
			& \quad+ (2\beta_k^2+\beta_k)\|x_{k-1}-x_*\|^2_2\\
			&\leq b_1\|x_k-x_*\|_2^2+b_2\|x_{k-1}-x_*\|_2^2,
		\end{aligned}
	\end{equation}
	where $b_1=2\beta_{\max}^2+3\beta_{\max}+1-\frac{1-2\eta\delta}{(1+\eta^2)K^2\delta^2}=2\beta_{\max}^2+3\beta_{\max}+1-c$ and $b_2=2\beta_{\max}^2+\beta_{\max}$.

	Noting that $\beta_{\max}< C_2(\eta, \gamma)\leq \frac{\sqrt{1+c}-1}{2}$,
	it is easy to verify that $b_1+b_2 = 4\beta_{\max}^2+4\beta_{\max}+1-c<1$ and $b_2>0$. Therefore, by Lemma \ref{eq: key_le}, we 
	readily derive 
	\begin{equation}\label{eq:main ideal}
		\|x_{k+1}-x_*\|_2^2\leq q^k(1+\xi)\|x_{0}-x_*\|_2^2,
	\end{equation} 
	where $q= \frac {b_{1}+ \sqrt {b_{1}^{2}+ 4b_{2}}}2$, $\xi = q- b_{1}$ and $b_{1}+ b_{2}\leq q< 1.$

	Next, we verify that \( q < d^2 \). In fact, we only need to show that \(d^4 - d^2b_1 - b_2 > 0\). 
	Substituting the expressions for \(b_1\) and \(b_2\), and using the conditions \(\beta_{\max} < C_2(\eta, \gamma)\leq \frac{-(1+3d^2)+\sqrt{(1+3d^2)^2+8(1+d^2)(d^4+d^2c-d^2})}{4(1+d^2)}\) and $K<C_1(\eta, \gamma)$, we can easily get that \(q < d^2\).
	
	On the other hand, we can deduce from \eqref{eq:main ideal} that 
	\begin{equation}
		\|x_{k+1}-x_*\|_2\leq \sqrt{q^k(1+\xi)}\|x_{0}-x_*\|_2\leq 2\|x_{0}-x_*\|_2\leq \frac{\delta_0}{2},
	\end{equation}
	That is to say $x_{k+1}\in B_{\delta_0/2}(x_*) \subseteq B_{\delta_0}(x_0)$. According to the iteration \eqref{eq:abnkam1}, we can write 
	\begin{align}\label{eq:which space}
		x_{k+1}-x_* &=x_k-x_*+\frac{\|x_k-x_{k-1}\|^2_2\|F_{\mathcal{J}_k}(x_k)\|^2_2}{\Delta_k}(F^{\prime}_{\mathcal{J}_k}(x_k))^TF_{\mathcal{J}_k}(x_k)\\
		&+\beta_k(x_k-x_*+x_*-x_{k-1}) \notag.
	\end{align}
	By means of Lemma \ref{le:case2}, \eqref{eq:which space} and $x_j \in B_{\delta_0/2}(x_*), \, x_j - x_* \in \ker(F'(x_*))^\perp$ for all $j\in \{0,1,\cdots,k\} $, we derive 
	\begin{equation}\nonumber
		\begin{aligned}
			x_{k+1}-x_*&\in \ker(F'(x_*))^{\perp}+Im(F'(x_k)^T)+\ker(F'(x_*))^{\perp}\\
			&=\ker(F'(x_*))^{\perp}+\ker(F'(x_k))^\perp\\
			&\subseteq \ker(F'(x_*))^{\perp}.
		\end{aligned}
	\end{equation}
	This completes the proof.
\end{proof}

We easily see from Theorem \ref{th:ideal} that the rate of convergence of ABNKAm (Algorithm \ref{alg:ABNKAm}) 
is given by $q= \frac {b_{1}+ \sqrt {b_{1}^{2}+ 4b_{2}}}2$, 
whereas the rate of convergence of 
the average block nonlinear Kaczmarz method, say the variant ABNK2 with an adaptive step size, 
is given by \cite{xiao2024averaging} $\rho_k=1-\frac{1-2\eta}{(1+\eta^2)\kappa^2(F_{\mathcal{J}_k}^{\prime}(x_k))}$, with 
$\rho_k\geq 1-\frac{1-2\eta}{1+\eta^2}=d$. We know from Theorem \ref{th:ideal} that $q< d^2<d$. 
This concludes that the rate of convergence of the averaged block Kaczmarz method with adaptive momentum 
is strictly smaller than that of the averaged block Kaczmarz method without adaptive momentum.

\subsection{Convergence analysis of Algorithm \ref{alg:ABNKAm2}}
We have derived in the subsection \ref{sec:con 1} the rate of convergence of Algorithm \ref{alg:ABNKAm} 
(cf.\,Theorem \ref{th:ideal}). We are now going to estimate the rate of convergence of 
Algorithm \ref{alg:ABNKAm2}. We first present a very important result for our subsequence analysis. 
\begin{lemma}\label{eq: key_le2}
	Fix \( E_1 = E_0 \geq 0 \) and let \( \{E_k\}_{k \geq 0} \) be a nonnegative real sequence. For each \( k \geq 1 \), the sequence satisfies one of the following two inequalities:
	(i) \( E_{k+1} \leq a_1 E_k + a_2 E_{k-1} \), where \( a_2 \geq 0 \), \( a_1 + a_2 < 1 \) and at least one of the coefficients \( a_1 \), \( a_2 \) is positive;
	(ii) \( E_{k+1} \leq a_3 E_k \), where \(a_3>0\).
	Let 
	\[
	t_{k-1}=\{\text{The  inequality (ii) occurs exactly }  t_{k-1}  \text{ times for any  }k \geq 1\}.
	\]
	Then the sequence \( \{E_k\}_{k \geq 0} \) satisfies the estimate 
	\[
	E_{k+1} \leq q^{k-t_k}(a_3+\xi)^{t_k} (1 + \xi) E_0 \text{ for all } k \geq 1,
	\]
	where $q = \frac{a_1 + \sqrt{a_1^2 + 4a_2}}{2}$ and $\xi = q - a_1 \geq 0$. Moreover, we have
	$q \geq a_1 + a_2$ and the equality holds if and only if $a_2 = 0$ (or $q = a_1$ and $\xi = 0$).
\end{lemma}

\begin{proof}
	We can easily find that \(\xi\) satisfies the following two properties: \(\xi\geq0\) and \(a_{2}\leq(a_{1}+\xi)\xi\). 
	Let the indices corresponding to the \( t_k \) occurrences of Case (ii) be denoted as \( i_1 < \dots < i_{t_k} \). Without loss of generality, we assume \( t_k < k + 1 \). Under these settings, the following relations hold:
	\begin{equation}\label{eq:unroll}
			E_{k+1} +\xi E_k \leq (a_1+\xi)E_k+a_2E_{k-1}\leq (a_1+\xi)(E_k+\xi E_{k-1})=q(E_k+\xi E_{k-1}).
	\end{equation}
	In conjunction with Case (ii) and recurrence relation \eqref{eq:unroll}, the following relations can be derived:
	$E_{k+1}\leq E_{k+1} +\xi E_k $ and
		\begin{align}\label{eq:step1}
			E_{k+1} +\xi E_k \leq q(E_k+\xi E_{k-1}) \leq q^{k-i_{t_k}+1}(E_{i_{t_{k}}}+\xi E_{i_{t_{k}}-1})\leq q^{k-i_{t_k}+1}(a_3+\xi)E_{i_{t_{k}}-1}.
		\end{align}
	
	To further carry out the recurrence for \eqref{eq:step1}, we must examine whether \( i_{t_k} - 1 = i_{t_k - 1} \) holds.  In fact, regardless of which scenario occurs, the following relation always holds:
	\begin{equation}
		\begin{aligned}\label{eq:step2}
			q^{k-i_{t_k}+1}(a_3+\xi)E_{i_{t_{k}}-1}\leq q^{k-i_{t_k-1}}(a_3+\xi)^2(E_{i_{t_{k}-1}-1}+\xi E_{i_{t_{k}-1}-2}).
		\end{aligned}
	\end{equation}
	Indeed, \eqref{eq:step2} stems from the following fact:
	\begin{equation}
		\begin{aligned}\label{eq:step3}
			&q^{k-i_{t_k}+1}(a_3+\xi)E_{i_{t_{k}}-1} \\
			&\leq
			\begin{cases}
				q^{k-i_{t_k-1}}(a_3+\xi)a_3E_{i_{t_{k}-1}-1}, &   i_{t_k} - 1 = i_{t_k - 1};\\
				q^{k-i_{t_k-1}}(a_3+\xi)(E_{i_{t_{k}-1}}+\xi F_{i_{t_{k}-1}-1}), & i_{t_k} - 1 \neq i_{t_k - 1};
			\end{cases}\\
			& \leq 
			\begin{cases}
				q^{k-i_{t_k-1}}(a_3+\xi)^2 E_{i_{t_{k}-1}-1}, &  i_{t_k} - 1 = i_{t_k - 1};\\
				q^{k-i_{t_k-1}}(a_3+\xi)^2E_{i_{t_{k}-1}-1}, & i_{t_k} - 1 \neq i_{t_k - 1};
			\end{cases}\\
			& \leq 
			\begin{cases}
				q^{k-i_{t_k-1}}(a_3+\xi)^2(E_{i_{t_{k}-1}-1}+\xi E_{i_{t_{k}-1}-2}), &  i_{t_k} - 1 = i_{t_k - 1};\\
				q^{k-i_{t_k-1}}(a_3+\xi)^2(E_{i_{t_{k}-1}-1}+\xi E_{i_{t_{k}-1}-2}), & i_{t_k} - 1 \neq i_{t_k - 1}.
			\end{cases}
		\end{aligned}
	\end{equation}
	Thus, by repeating the reasoning from \eqref{eq:unroll} to \eqref{eq:step2}, we can obtain the following relation:
	\begin{equation}\label{eq:key }
		\begin{aligned}
			E_{k+1} \leq q^{k-t_k}(a_3+\xi)^{t_k}(E_{1}+\xi E_{0})\leq q^{k-t_k}(a_3+\xi)^{t_k}(1+\xi)E_0.
		\end{aligned}
	\end{equation}
	This completes the proof.
\end{proof}

Recalling the fact that Algorithm \ref{alg:ABNKAm2} updates the approximate solution using \eqref{eq:abnkam2},  
we next introduce the following convergence result that is needed for our subsequent analysis. It is based on a modified version of Theorem 1 in \cite{tan2024nonlinear} and Theorem 3.2 in \cite{xiao2024averaging}, but tailored for our special purposes here.
\begin{lemma}\cite{tan2024nonlinear, xiao2024averaging}\label{le:case2}
	Suppose  $F(x)$ satisfies Assumption \ref{as:asc} and $\ker(F'(x_*))\subseteq \ker(F'(x))$  for all $x\in B_{\delta_0/2}(x_*)$,
	where $x_*$ is the unique \( x_0 \)-minimum-norm solution of $F(x)=0$. Then 
	
	(i) The iterative sequence \( \{x_k\} \) generated by 
	\[
	x_{k+1} = x_k-\frac{\|F_{\mathcal{J}_k}(x_k)\|^2_2}{\|(F^{\prime}_{\mathcal{J}_k}(x_k))^TF_{\mathcal{J}_k}(x_k)\|^2_2}(F^{\prime}_{\mathcal{J}_k}(x_k))^TF_{\mathcal{J}_k}(x_k)
	\]
	satisfies the inclusion relation
	$x_k \in B_{\delta_0/2}(x_*) \subseteq B_{\delta_0}(x_0) \text{ for all } k \geq 0.$
	
	(ii) If $F(x_k)\neq0$  for some $k\geq0$, we have  $ x_{k+1}-x_*\in \ker(F'(x_*))^\perp$
	and
	\begin{equation}
		\|x_{k+1}-x_*\|_2^2 \leq \left( 1-\frac{1-2\eta}{(1+\eta^2)\kappa^2(F_{\mathcal{J}_k}^{\prime}(x_k))} \right) \|x_k-x_*\|_2^2.
	\end{equation}
\end{lemma}

We next present the rate of convergence of the practical variant of ABNKAm (Algorithm \ref{alg:ABNKAm2}), 
which also exhibits exponential convergence.
For the sake of simplicity,	we first introduce some constants.	
Let 
\[
\varepsilon_{1}=\min\Big\{ \frac{\eta(1+\eta^2)+\sqrt{\eta^2(1+\eta^2)^2+(1+\eta^2)(1-2\eta^2-4\eta^3)}}{1-2\eta^2-4\eta^3}, \frac{-\eta+\sqrt{2\eta^2+2\eta}}{\eta^2+2\eta} \Big\}.
\]
 We can check that $\varepsilon_{1}\in(0, 1)$.
For any $\gamma \in(\varepsilon_{1}/2, \varepsilon_{1})$, we define 
$C_3(\eta, \gamma):=\min\left\{ \sqrt{\frac{1-2\eta\gamma+(1-2\eta)\gamma^2}{(1+\eta^2)\gamma^2}}, \sqrt{\frac{(1-2\eta\gamma)(1+\eta^2)}{(1-2\eta^2-4\eta^3)\gamma^2}} \right\}$,
and 
\begin{align*}
	C_4(\eta, \gamma) &:= \min \Big\{ \frac{-(1+3d^2)+\sqrt{(1+3d^2)^2+8(1+d^2)(d^4+d^2c-d^2)}}{4(1+d^2)}, \frac{\sqrt{1+c}-1}{2},\\
	&\quad \frac{\sqrt{(4-3b_3)^2+8(2-b_3)(1-b_3)(c-b_3)}-(4-3b_3)}{4(2-b_3)} \Big\}.
\end{align*}
We can easily verify that \( C_3(\eta, \gamma) > 1 \) and \( C_4(\eta, \gamma) > 0 \).

\begin{theorem}\label{thm:them2}
	Let $F(x)$ satisfy Assumption \ref{as:asc} with $x_*$ being the unique $x_0$-minimum-norm solution of $F(x)=0$. Assume $\ker(F'(x_*)) \subseteq \ker(F'(x))$ for all $x \in B_{\delta_0/2}(x_*)$, and let $\{x_k\}$ be the sequence generated by Algorithm \ref{alg:ABNKAm2}.
	Assume there exists a constant $\gamma \in(\varepsilon_{1}/2, \varepsilon_{1})$ such that the following conditions hold for all \(k\):
	\begin{enumerate}
		\item[(C1)] $\kappa(F'_{\mathcal{J}_k}(x_k)) \leq K$ for some $K \in [1, C_3(\eta, \gamma))$,
		\item[(C2)] $\beta_k \in [0, \beta_{\max})$ for some $\beta_{\max} \in (0, C_4(\eta, \gamma))$,
		\item[(C3)] $\sin^2\theta_k \leq \gamma$,
	\end{enumerate}
	then the following conclusions hold:
	
	(i) $ \{x_k\} \subseteq B_{\delta_0/2}(x_*) \subseteq B_{\delta_0}(x_0) $ for all $k \geq 0$.
	
	(ii) If $F(x_k) \neq 0$ for some $k$, then $x_{k+1} - x_* \in \ker(F'(x_*))^\perp$ and
	\[
	\|x_{k+1} - x_*\|_2^2 \leq Q(k, t_k) \|x_0 - x_*\|_2^2,
	\]
	where $t_k$ is defined as in Lemma \ref{eq: key_le2} and $Q(k, t_k)$ is given by $Q(k, t_k) = q^{k - t_k}(b_3 + \xi)^{t_k}(1 + \xi)$ and
	\begin{align*}
		&q = \frac{b_1 + \sqrt{b_1^2 + 4b_2}}{2},  \xi = q - b_1,  b_1 + b_2 \leq q < d^2 < 1, b_3 + \xi < 1\\
		&b_1 = 1 + 3\beta_{\max} + 2\beta_{\max}^2 - \frac{1 - 2\eta\delta}{(1+\eta^2)K^2\delta^2}, 
		b_2 = 2\beta_{\max}^2 + \beta_{\max},  b_3 = 1 - \frac{1 - 2\eta}{(1+\eta^2)K^2}.
	\end{align*}
	
	(iii) For sufficiently large $k$, the rate of convergence $Q(k, t_k)$ admits the following upper bound estimates depending on 
	the range of $t_k$:
	\begin{itemize}
		\item $Q(k, t_k) \le q$ for $1 \leq t_k \ll k + 1$;
		
		\item $Q(k, t_k) \le \sqrt{q(b_3 + \xi)} < d$ for $1 < t_k < \lfloor (k+1)/2 \rfloor$; 
		
		\item $Q(k, t_k) \le b_3 + \xi$ for $\lfloor (k+1)/2 \rfloor \leq t_k \leq k + 1$. 
		
	\end{itemize}
\end{theorem}

\begin{proof}
	We apply the mathematical induction.
	From Assumption \ref{as:asc}(iv) and Lemma \ref{le:case2}, we have $x_0\in B_{\delta_0/4}(x_*)\subseteq B_{\delta_0}(x_0),x_1\in B_{\delta_0/2}(x_*)\subseteq B_{\delta_0}(x_0)$ and $x_0-x_*,x_1-x_*\in \ker(F'(x_*))^{\perp}$.
	We now assume without loss of generality that $x_j \in B_{\delta_0/2}(x_*) \subseteq B_{\delta_0}(x_0),\, x_j - x_* \in \ker(F'(x_*))^\perp$ for all $j\in \{0,1,\cdots,k\}$ hold. Then it remains to show that \( x_{k+1} \in B_{\delta_0/2}(x_*) \subseteq B_{\delta_0}(x_0) \) and \( x_{k+1} - x_* \in \ker(F'(x_*))^\perp \).
	
	Because the iterative formula \eqref{eq:abnkam2} involves two different situations, namely (i) $\Delta_k \geq\varepsilon$ and $\beta_k\in(0,\beta_{\max})$ and (ii) $\Delta_k <\varepsilon$ or $\beta_k\notin (0,\beta_{\max})$. We will discuss each case separately.
	
	(i) $\Delta_k \geq\varepsilon$ and $\beta_k\in(0,\beta_{\max})$. 
	Using the iterative formula \eqref{eq:abnkam1} at the \( k \)-th step, 
	we follow the proof of Theorem \ref{th:ideal} to deduce 
	\begin{equation}\label{eq:c1}
		\begin{aligned}
			\|x_{k+1}-x_*\|_2^2
			&\leq (1+3\beta_k+2\beta_k^2-\frac{(1-2\eta+\cot^2\theta_k)}{(1+\eta^2)\kappa^2(F_{\mathcal{J}_k}^{\prime}(x_k))\sin^2\theta_k})\|x_k-x_*\|_2^2\\
			& \quad+ (2\beta_k^2+\beta_k)\|x_{k-1}-x_*\|^2_2\\
			&\leq b_1\|x_k-x_*\|_2^2+b_2\|x_{k-1}-x_*\|_2^2,
		\end{aligned}
	\end{equation}
	where $b_1=2\beta_{\max}^2+3\beta_{\max}+1-\frac{1-2\eta\delta}{(1+\eta^2)K^2\delta^2}=2\beta_{\max}^2+3\beta_{\max}+1-c$ and $b_2=2\beta_{\max}^2+\beta_{\max}$. Noting that $0<\beta_{\max}< \frac{\sqrt{1+c}-1}{2}$
	it is easy to verify that $b_1+b_2 = 4\beta_{\max}^2+4\beta_{\max}+1-c<1$ and $b_2>0$. 
	
	(ii) $\Delta_k <\varepsilon$ or $\beta_{k} \notin (0,\beta_{\max})$. 
	In this case, we have 
	\[
	x_{k+1}=x_k-\frac{\|F_{\mathcal{J}_k}(x_k)\|^2_2}{\|F^{\prime}_{\mathcal{J}_k}(x_k)^TF_{\mathcal{J}_k}(x_k)\|^2_2}F^{\prime}_{\mathcal{J}_k}(x_k)^TF_{\mathcal{J}_k}(x_k).
	\]
	Then according to Lemma \ref{le:case2}, we know (with $b_3=1-\frac{1-2\eta}{(1+\eta^2)K^2}$)
	\begin{equation}\label{eq:c2}
		\begin{aligned}
			\|x_{k+1}-x_*\|_2^2
			&\leq (1-\frac{1-2\eta}{(1+\eta^2)\kappa^2(F_{\mathcal{J}_k}^{\prime}(x_k))})\|x_k-x_*\|_2^2\leq b_3\|x_k-x_*\|_2^2\,.
		\end{aligned}
	\end{equation}
	Combining \eqref{eq:c1}, \eqref{eq:c2}, Theorem \ref{th:ideal} and Lemma \ref{eq: key_le2}, we obtain 
	the following estimate:
	\[
	\|x_{k+1}-x_*\|_2^2\leq q^{k-t_k}(b_3+\xi)^{t_k} (1 + \xi) \|x_{0}-x_*\|_2^2,
	\]
	where $q= \frac {b_{1}+ \sqrt {b_{1}^{2}+ 4b_{2}}}2$, $\xi = q- b_{1}$, $b_{1}+ b_{2}\leq q< d^2<1$, and $b_1=1+3\beta_{\max}+2\beta_{max}^2-\frac{1-2\eta\delta}{(1+\eta^2)K^2\delta^2},\,b_2=2\beta_{\max}^2+\beta_{\max}$, $b_3=1-\frac{1-2\eta}{(1+\eta^2)K^2}$.
	
	Next, we verify that \( b_3 + \xi < 1 \). Indeed, we only need to verify that \( b_2+(1-b_3)b_1-(1-b_3)^2<0 \). 
	Substituting the expressions for \(b_1\) and \(b_2\), and using the conditions \( \beta_{\max} < C_4(\eta, \gamma)\leq \frac{-(4-3b_3)+\sqrt{(4-3b_3)^2+8(2-b_3)(1-b_3)(c-b_3)}}{4(2-b_3)} \) and $ K^2 <C_3(\eta, \gamma)^2\leq\frac{1-2\eta\gamma+(1-2\eta)\gamma^2}{(1+\eta^2)\gamma^2}$, we can easily verify that \( b_3 + \xi < 1 \).
	
	On the other hand, we easily see 
	\begin{equation}
		\|x_{k+1}-x_*\|_2\leq \sqrt{q^{k-t_k}(b_3+\xi)^{t_k} (1 + \xi)} \|x_{0}-x_*\|_2\leq \frac{\delta_0}{2}
	\end{equation}
	that is to say $x_{k+1}\in B_{\delta_0/2}(x_*) \subseteq B_{\delta_0}(x_0)$.
	
	Similar to the proof of Theorem \ref{th:ideal}, base on \eqref{eq:abnkam2},  Lemma \ref{le:case2} and $x_j \in B_{\delta_0/2}(x_*) \subseteq B_{\delta_0}(x_0),\ x_j - x_* \in \ker(F'(x_*))^\perp$  for all $j\in \{0,1,\cdots,k\} $, we can derive 
	\begin{equation}\nonumber
		\begin{aligned}
			x_{k+1}-x_*&\in \ker(F'(x_*))^{\perp}+Im(F'(x_k)^T)\\
			&=\ker(F'(x_*))^{\perp}+\ker(F'(x_k))^\perp\\
			&\subseteq \ker(F'(x_*))^{\perp}.
		\end{aligned}
	\end{equation}
	This completes the proof.
\end{proof}

\begin{remark}
	Although both Theorems \ref{th:ideal}-\ref{thm:them2} require \( \beta_{\text{max}} \) to stay 
	in some certain range, we have observed very favorable results by simply 
	setting \( \beta_{\text{max}} = +\infty \) in all our numerical experiments (see Section \ref{sec:ne}).
\end{remark}

According to Theorem \ref{thm:them2}, we know that the rate of convergence of Algorithm \ref{alg:ABNKAm2} depends 
on the concrete values of \( t_k \). If \( t_k \ll  k+1 \), we always have \( \rho < d^2 \). This condition is relatively easy to satisfy, as evidenced by the numerical experiments (see Figure \ref{fig:CHerror}) in Section \ref{sec:ne}. If \( t_k < \lfloor (k+1)/2 \rfloor \), we have  \( \rho\le \sqrt{q(b_3 + \xi)} <d \). 
If \( t_k > \lfloor (k+1)/2 \rfloor \), the upper bound estimate of the rate we have provided is \( b_3 + \xi \).  But at the extreme case (i.e., \( t_k = k+1 \)), the ABNKAm degenerates to the ABNK2, and it is obvious to see that their rates of convergence are the same in this scenario. Therefore, we can always conclude that 
the rate of convergence of the ABNKAm is less than or equal to that of the ABNK2.

\section{Numerical experiments}\label{sec:ne}
In this section, we present several representative numerical examples to verify 
the effectiveness of the proposed ABNKAm method (Algorithm \ref{alg:ABNKAm2}). The computational performance is evaluated in two metrics: the number of iterations (denotes as IT) and the CPU time in seconds (denotes as CPU). To quantitatively assess efficiency improvements, we further define the speed-up ratio as
\begin{equation}\label{eq:su}
	\mathrm{SU}=\frac{\text{CPU of other method}}{\text{CPU of ABNKAm}},
\end{equation}
where other methods include the nonlinear Kaczmarz method and its several variants. All numerical experiments adopt the following termination criteria:
$\|r_k\|_2\leq \tau_a + \tau_r \|r_0\|_2$ proposed in \cite{kelley2003solving}
or the number of iterations exceeds $10^{5}$, where $\tau_a=10^{-6}$ and $\tau_r=10^{-8}$.

The numerical experiments are performed in PYTHON on an AMD Ryzen 9 7945HX with Radeon Graphics computer with 32.0 GB of RAM. CPU time refers to the duration from the initial value until the termination criterion is met for the algorithm. 
In all experiments, we set $\beta_{\max} = +\infty$ and $\varepsilon = 10^{-16}$ in Algorithm \ref{alg:ABNKAm2}. In the tables below, the symbol ``-'' indicates that the maximum number of iterations was exceeded or that memory was insufficient.

We utilize several singular problems to test the proposed methods and conduct comparisons with other methods. Specifically, the proposed methods are compared with:
\begin{itemize}
	\item
	the averaging block nonlinear Kaczmarz method with adaptive stepsize (denoted as ABNK2 \cite{xiao2024averaging})
	\item 
	the residual-based block capped nonlinear Kaczmarz (denoted as RB-CNK \cite{zhang2024greedy})
	\item 
	the maximum residual block nonlinear Kaczmarz (denoted as MRBNK \cite{zhang2023maximum})
	\item
	the greedy randomized Kaczmarz with momentum (denoted as NGRKm \cite{liu2025greedy}) 
	\item
	the maximum residual nonlinear Kaczmarz (denoted as MRNK \cite{zhang2023maximum}) 
	\item
	the nonlinear uniformly randomized Kaczmarz (denoted as NURK \cite{wang2022nonlinear})
\end{itemize}

For both the MRBNK and RB-CNK methods, we employ the LSQR algorithm \cite{paige1982lsqr} to approximate the pseudoinverse computation required in \eqref{eq:mrbnk}. The primary distinction between RB-CNK and MRBNK lies in their greedy selection criteria; detailed explanations can be found in \cite{zhang2024greedy}. It is also worth noting that the ABNK2, MRBNK, and ABNKAm methods all utilize the greedy criterion specified in \eqref{eq:greed}. For the control parameter $\theta$ in \eqref{eq:greed}, we adopt the values from the literature for those benchmark problems; for new and non-benchmark problems, 
$\theta$ was set to 0.5. As for the momentum parameter $\beta$ in the NGRKm method, we select appropriate values to optimize its convergence performance.

\begin{pro}\label{P1}
	The modified Rosenbrock problem \cite{friedlander1997solving}:
	\begin{equation}\label{p1}
		\begin{cases}
			&F_{k}(x)  =\frac{1}{1+\exp(-x_{k})}-0.73, \quad\mathrm{mod}(k,2)=1, \notag\\
			&F_{k}(x) =10(x_{k}-x_{k-1}^2), \quad\mathrm{mod}(k,2)=0,\quad 
			k=1,\dots, n. 
		\end{cases}  
	\end{equation}
	The initial value $x_0=(x^1,\ldots,x^n)\in\mathbb{R}^n$ was used in this experiment, where $x^{k}=-1.8$ for $\mathrm{mod}(k,2)=1$ and $x^{k}=-1$ for  $\mathrm{mod}(k,2)=0$. 
\end{pro}

\begin{table}[htbp]
	\centering
	\begin{threeparttable}  
		\caption{IT and CPU of different methods for Modified Rosenbrock Problem}
		\label{tab:Modified Rosenbrock}
		\setlength{\tabcolsep}{2.5pt}  
		\begin{tabular}{@{}llccccccc@{}}
			\toprule
			\multirow{2}{*}{m} & \multirow{2}{*}{Metric} & \multicolumn{7}{c}{Algorithms} \\ 
			\cmidrule(lr){3-9}
			& & \multicolumn{1}{c}{ABNKAm} & \multicolumn{1}{c}{ABNK2} & \multicolumn{1}{c}{RB-CNK}  & \multicolumn{1}{c}{MRBNK} & \multicolumn{1}{c}{NGRKm} & \multicolumn{1}{c}{MRNK} & \multicolumn{1}{c}{NURK} \\ 
			\midrule
			\multirow{3}{*}{$1\cdot10^3$} & IT & 9 & 2913 & 124  & 88 & 61567 & 61736 & -\\ 
			& CPU & 0.0162 & 7.4616 & 0.2023 & 0.1593 & 12.265 & 10.306 & - \\ 
			&$\mathrm{SU}$ &  & 460.6 & 12.5  & 9.8 & 757.1 & 636.2 & - \\ 
			\hline
			\multirow{3}{*}{$1\cdot10^4$} & IT & 9 & 2913 & 124  & 90 & - & - & -  \\ 
			& CPU & 0.1322 & 85.795 &  1.9305  & 1.3946 & - & - & - \\ 
			&$\mathrm{SU}$ &  & 649.0 & 14.6 & 10.5 & - & - & -  \\ 
			\hline
			\multirow{3}{*}{$1\cdot10^5$} & IT & 9 & 2913 & 126 & 90 & - & - & -\\ 
			& CPU & 1.1881 & 847.16 & 21.354 & 15.211 & - & - & - \\ 
			&$\mathrm{SU}$ &  & 713.0 & 18.0 & 12.8 & - & - & - \\ 
			\hline
			\multirow{3}{*}{$1\cdot10^6$} & IT & 9 & - & 126 & 90 & - & - & -\\ 
			& CPU & 9.6079 & -  & 210.52 & 153.67 & - & - & - \\ 
			&$\mathrm{SU}$ &  & - & 21.9 & 16.0 & - & - & - \\ 
			\bottomrule
		\end{tabular}
	\end{threeparttable} 
\end{table}

\begin{figure}[htbp]
	\centering
	\begin{minipage}[t]{0.45\textwidth}
		\includegraphics[width=\textwidth, height=4.5cm]{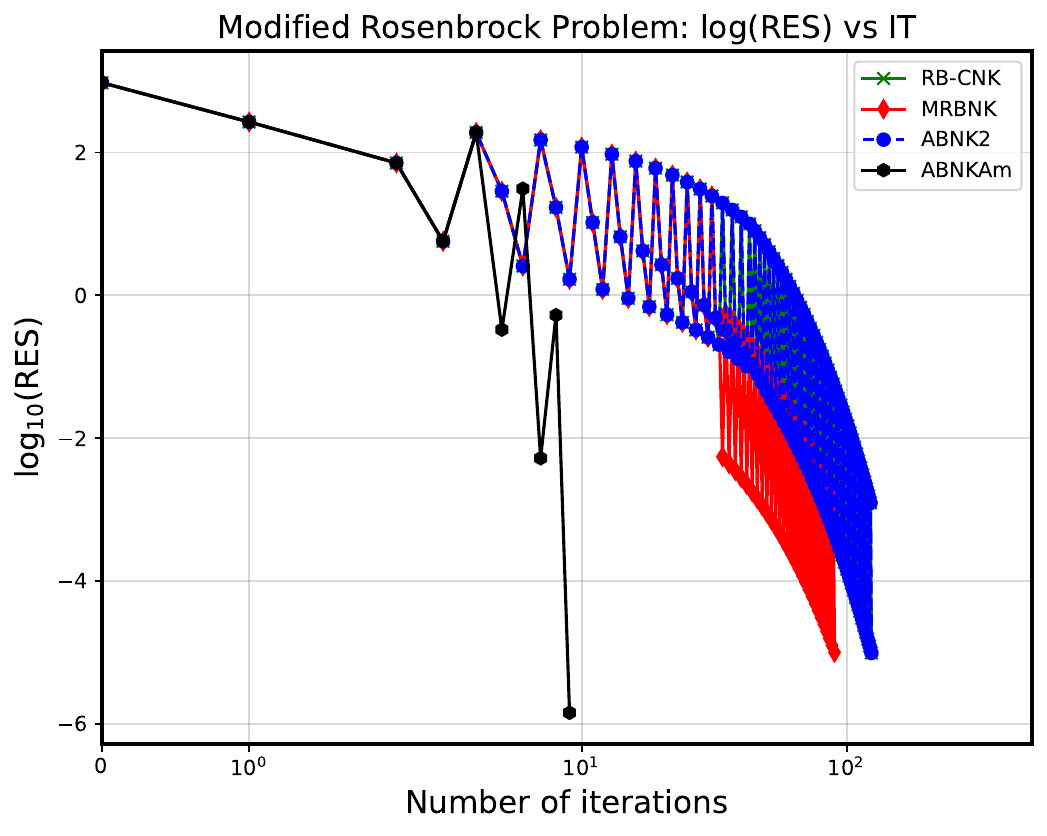}
	\end{minipage}
	\hfill
	\begin{minipage}[t]{0.45\textwidth}
		\includegraphics[width=\textwidth, height=4.5cm]{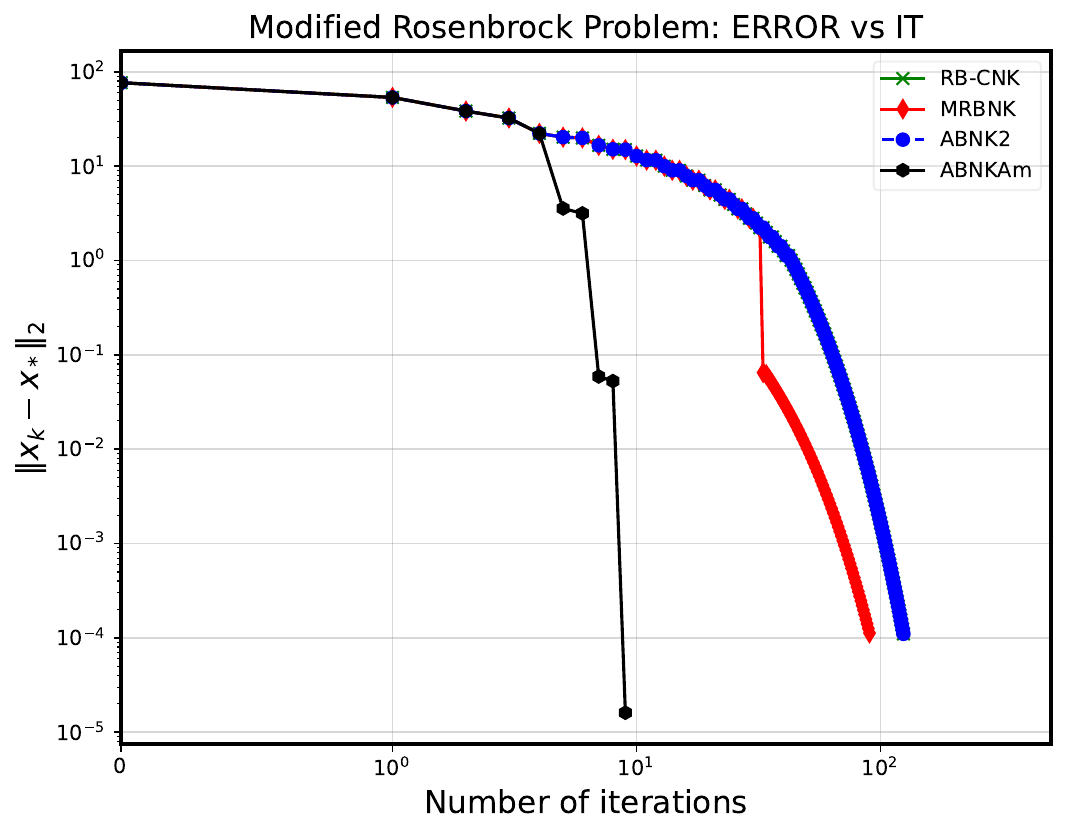}
	\end{minipage}
	\caption{Curves of ERROR and RES against IT for $m=1000$, Modified Rosenbrock Problem}
	\label{fig:MR1}
\end{figure}

Table \ref{tab:Modified Rosenbrock} provides a quantitative comparison of the performance of different methods across various problem scales (with $m=10^3, 10^4, 10^5, 10^6$), including iteration count (IT), CPU time, and speed-up ratio (SU). 
The results demonstrate that the proposed ABNKAm method outperforms all competing approaches in both CPU time 
and iteration count.
We can clearly observe that ABNK2 (scheme \eqref{eq:abnk}) can reduce per-iteration cost compared to the pseudoinverse-based MRBNK (scheme \eqref{eq:mrbnk}), this reduction fails to compensate for the increased number of iterations 
required. 
In contrast, ABNKAm achieves significant improvement over MRBNK in both iteration count and CPU time.
Furthermore, as shown in Table \ref{tab:Modified Rosenbrock}, the speed-up ratio of ABNKAm 
relative to other approaches continues to grow with problem scales, highlighting its superior scalability.

Figure \ref{fig:MR1} (left) reveals oscillatory behavior in the residual curves of all methods. This phenomenon can be attributed to the "seesaw effect" \cite{li2025global}: when an index set $\mathcal{J}_k$ with larger residuals is selected via the greedy criterion \eqref{eq:greed} and processed in the current iteration, the residuals of previously smaller components may increase. Mitigating this effect is crucial for improving the convergence behavior of Kaczmarz-type methods. The proposed ABNKAm method partially addresses this issue.
Moreover, although Kaczmarz-type methods may exhibit intermittent residual increases, the sequence of approximate solutions consistently approaches the true solution $x_*$ (taken as the high-precision numerical solution obtained by the ABNKAm method), as evidenced in Figure \ref{fig:MR1} (right).

\begin{pro} \label{P2}
	The extended Cragg levy problem \cite{lukvsan1994inexact}:
	\begin{equation}
		\begin{cases}
			&F_{k}(x)  = (\exp(x_k)-x_{k+1})^2, \quad\mathrm{mod}(k,4)=1, \notag\\
			&F_{k}(x) =10(x_{k}-x_{k+1})^3, \quad\mathrm{mod}(k,4)=2, \notag\\
			&F_{k}(x) =\tan^2(x_k-x_{k+1}), \quad\mathrm{mod}(k,4)=3,\notag\\
			&F_{k}(x) = x_k-1, \quad\mathrm{mod}(k,4)=0,
			\quad 
			k=1,\dots, n. 
		\end{cases}  
	\end{equation}
	The initial value $x_0=(x^1,\ldots,x^n)\in \mathbb{R}^n$ is set to 
	$x^{k}=1$ for $\mathrm{mod}(k,4)=1$, and $x^{k}=2$ for  $\mathrm{mod}(k,4)\neq 0$. 
\end{pro}

Numerical results for this example are summarized in Table \ref{tab:Extended Cragg Levy Problem}. As shown in the table, row-action methods such as NGRKm, MRNK, and NURK fail to converge, whereas all block-oriented methods successfully reach the prescribed tolerance. Among the convergent methods, the proposed ABNKAm achieves significantly fewer iterations and less 
computational time than its counterparts.

Notably, we can see that ABNK2 (scheme \eqref{eq:abnk}) performs worse than MRBNK 
but better than RB-CNK in both iteration count and CPU time. 
Although ABNK2 reduces per-iteration cost, the resulting increase in iterations leads to longer total runtime 
compared to MRBNK. The superior performance of MRBNK over RB-CNK further confirms the effectiveness 
of the greedy criterion \eqref{eq:greed}.
In contrast, the proposed ABNKAm method converges faster than all other block methods. 
For a problem size of $10^6$, it achieves speed-up factors of 3.0, 3.4, and 2.7 over ABNK2, RB-CNK, and MRBNK, respectively.

\begin{table}[h]
	\centering
	\begin{threeparttable}
		\caption{IT and CPU of different methods for Extended Cragg Levy Problem}
		\label{tab:Extended Cragg Levy Problem}
		\setlength{\tabcolsep}{2.5pt}  
		\begin{tabular}{@{}llccccccc@{}}
			\toprule
			\multirow{2}{*}{m} & \multirow{2}{*}{Metric} & \multicolumn{7}{c}{Algorithms} \\ 
			\cmidrule(lr){3-9}
			& & \multicolumn{1}{c}{ABNKAm} & \multicolumn{1}{c}{ABNK2} & \multicolumn{1}{c}{RB-CNK} & \multicolumn{1}{c}{MRBNK} & \multicolumn{1}{c}{NGRKm} & \multicolumn{1}{c}{MRNK} & \multicolumn{1}{c}{NURK} \\ 
			\midrule
			\multirow{3}{*}{$1\cdot10^3$} & IT & 169 & 322 & 399 & 288& - & - & -  \\ 
			& CPU & 0.2066 & 0.5131 & 0.6483 & 0.5018 & - & - & -\\ 
			&$\mathrm{SU}$ &  & 2.5 & 3.1 & 2.4 &-& - & -\\ 
			
			\hline
			\multirow{3}{*}{$1\cdot10^4$} & IT & 184 & 350 & 428 & 313 & - & - & - \\ 
			& CPU & 1.9546 & 5.0999 & 5.9925 & 4.8753 & - & - & - 
			\\ 
			&$\mathrm{SU}$ &  & 2.6 & 3.1 & 2.5 & - & - & - \\ 		
			\hline
			\multirow{3}{*}{$1\cdot10^5$} & IT & 186 & 375 & 460 & 330 & - & - & - \\ 
			& CPU & 18.212 & 51.338 & 59.282 & 46.209 & - & - & -\\ 
			&$\mathrm{SU}$ &  & 2.8 & 3.3 & 2.5 & - & - & - \\
			\hline
			\multirow{3}{*}{$1\cdot10^6$} & IT & 185 & 386 & 470 & 332 & - & - & - \\ 
			& CPU & 160.58 & 489.14 & 539.66 & 436.98 & - & - & -\\ 
			&$\mathrm{SU}$ &  & 3.0 & 3.4 & 2.7 & - & - & - \\
			\bottomrule
		\end{tabular}
	\end{threeparttable}
\end{table}

Next, we will test some nonsingular problems to verify 
the performance of the proposed methods,  in comparison with other methods. Since the row methods are significantly less efficient than the block methods, we only show the results with the NGRKm method.

\begin{pro}\label{P3}
	The Chandrasekhar H-equation \cite{babajee2010analysis}:
	\begin{equation}
		F(H, c)(u)=H(u)- \Big( 1-\frac{c}{2}\int_0^1\frac{u H(v)dv}{u+v} \Big)^{-1}=0.
	\end{equation}
	This problem originates from radiative transfer theory, where the integrals are discretized via the composite midpoint rule:
	$ \int_0^1f(t)dt\approx\frac{1}{m}\sum_{j=1}^mf(t_j)$,
	where $t_i=(i-\frac{1}{2})/m$ for $1\leq i\leq m$, and $m$ denotes the number of integration nodes. Consequently, the following discrete problem is solved:
	\begin{equation}\label{eq:dis}
		F_i(u, c)=u_i- \Big(1-\frac{c}{2m}\sum_{j=1}^m\frac{t_iu_j}{t_i+t_j} \Big)^{-1},
	\end{equation}
	we set $c=0.9$, then the discrete problem \eqref{eq:dis} yields a nonlinear systems of equations, where $m$ represents the number of nonlinear equations.
\end{pro}

\begin{figure}[htbp]
	\centering
	
	\begin{minipage}[t]{0.45\textwidth}
		\includegraphics[width=\textwidth, height=4.5cm]{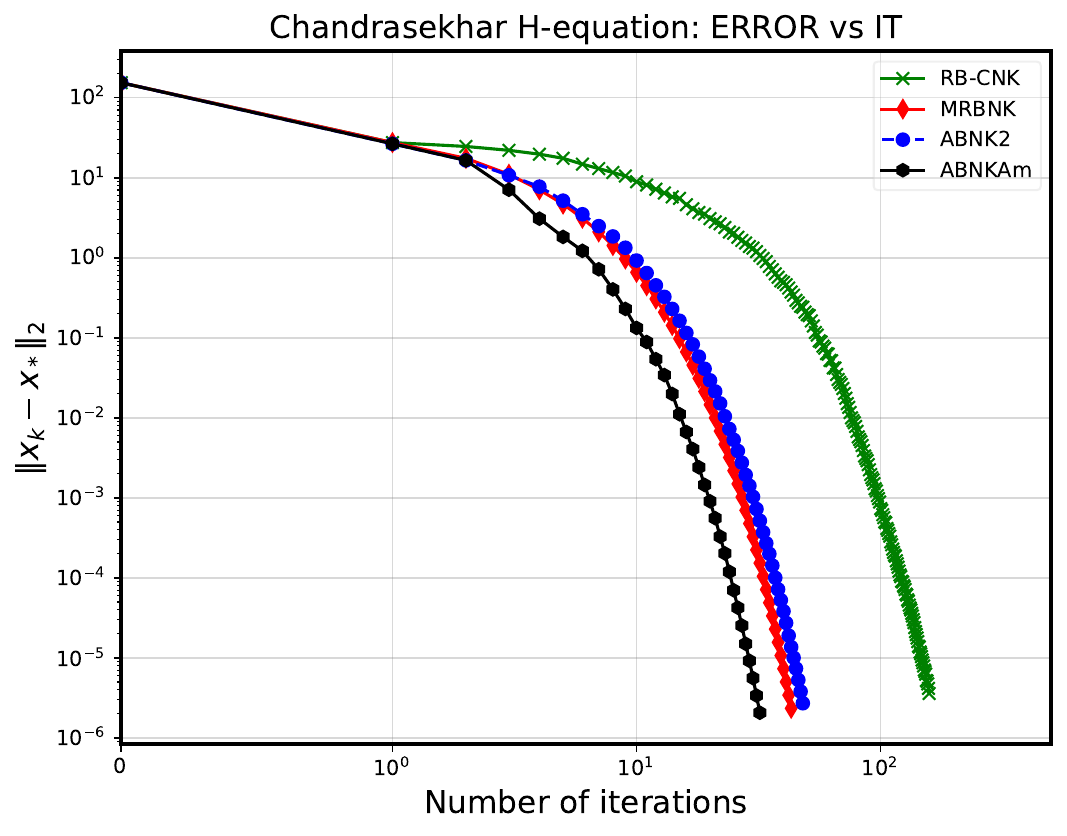}
	\end{minipage}
	\hfill %
	\begin{minipage}[t]{0.45\textwidth}
		\includegraphics[width=\textwidth, height=4.5cm]{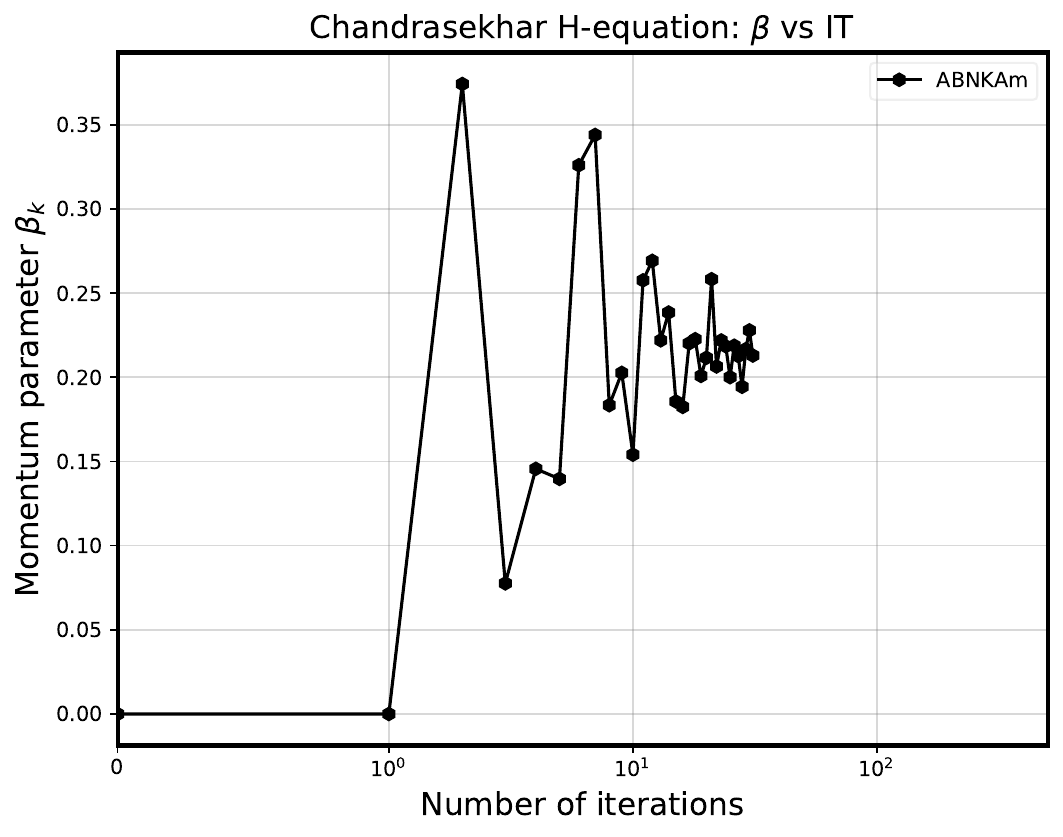}
	\end{minipage}
	\caption{Chandrasekhar H-equation ($m=10^4$). 
		Left: Curve of error versus IT; 
		Right; Curve of momentum parameter versus IT.} 
	\label{fig:CHerror}
\end{figure}

\begin{table}[h]
	\centering
	\begin{threeparttable}
		\caption{IT and CPU of different methods for Chandrasekhar H-equation}
		\label{tab:H-equation}
		\setlength{\tabcolsep}{3.5pt}  
		\begin{tabular}{@{}llccccc@{}}
			\toprule
			\multirow{2}{*}{m} & \multirow{2}{*}{Metric} & \multicolumn{5}{c}{Algorithms} \\ 
			\cmidrule(lr){3-7}
			& & \multicolumn{1}{c}{ABNKAm} & \multicolumn{1}{c}{ABNK2} & \multicolumn{1}{c}{RB-CNK} & \multicolumn{1}{c}{MRBNK} & \multicolumn{1}{c}{NGRKm}  \\ 
			\midrule
			\multirow{3}{*}{$1\cdot10^3$} & IT & 30 & 45 & 155 & 41 & 16598 \\ 
			& CPU & 0.7089 & 1.1037 & 6.3928 & 4.2516 & 239.65\\ 
			&$\mathrm{SU}$ &  &  1.6 & 9.0 & 6.0 & 338.1  \\ 
			
			\hline
			\multirow{3}{*}{$5\cdot10^3$} & IT & 31 & 47 & 147 & 47 & 85595 \\ 
			& CPU & 5.9860 & 9.3224 & 56.881 & 36.042 & 10211   \\ 
			&$\mathrm{SU}$ &  & 1.6 & 9.5 & 6.0 & 1705.8  \\  
			
			\hline
			\multirow{3}{*}{$1\cdot10^4$} & IT & 32 & 48 & 158 & 43 & - \\ 
			& CPU & 14.822 & 23.373 & 146.78 & 92.940 & - \\ 
			&$\mathrm{SU}$ &  & 1.6 & 10.0 & 6.3 & -   \\ 
			
			\hline
			\multirow{3}{*}{$5\cdot10^4$} & IT & 33 & 48 & 169 & 43 & -  \\ 
			& CPU & 293.03 & 497.30 & 3127.8 & 1874.4 & -  \\ 
			&$\mathrm{SU}$ &  & 1.7 & 10.7 & 6.4 & -  \\  
			
			\hline
			\multirow{3}{*}{$1\cdot10^5$} & IT & 33 & 49 & 167 & 43 & - \\ 
			& CPU & 1182.3 & 2213.5 & 12713 & 7984.9 & - \\ 
			&$\mathrm{SU}$ &  & 1.9 & 10.8 & 6.8 & -  \\ 
			
			\bottomrule
		\end{tabular}
	\end{threeparttable}  
\end{table}

The initial value is set to $x_0 = (0, \dots, 0) \in \mathbb{R}^n$, and numerical results are presented in Table \ref{tab:H-equation}. The proposed ABNKAm method outperforms all other methods in both iteration count and CPU time. Although the problem structure is simple, its dense Jacobian matrix leads to high computational and memory costs.

As shown in Table \ref{tab:H-equation}, the row-action method NGRKm requires substantially more computational time than block-oriented methods, even with momentum acceleration. Among block methods, those avoiding pseudoinverse computations (ABNK2) are notably faster than pseudoinverse-based approaches, confirming the advantage of pseudoinverse-free schemes for problems with dense Jacobians.
Furthermore, the proposed ABNKAm converges faster than the momentum-free ABNK2 method in both iteration and runtime. The speed-up ratio (SU) improves with problem scales. At a scale of $10^5$, ABNKAm achieves speed-up factors of approximately 2.0 over ABNK2 and 7.0 over MRBNK.

Figure \ref{fig:CHerror} provides additional insight into the convergence behavior. The per-iteration progress, measured by $\|x_k - x_*\|_2$, is significantly larger for ABNKAm than for other block Kaczmarz methods, aligning with our theoretical analysis in Section \ref{sec:convergence}. Notably, while ABNKAm and ABNK2 behave identically in the first two iterations (when $\beta = 0$), ABNKAm shows markedly greater progress per iteration from the third iteration onward (with $\beta \neq 0$), clearly demonstrating the acceleration effect of the momentum term.

\begin{pro}\label{P4}
	The augmented Rosenbrock problem \cite{friedlander1997solving}:
	\begin{equation}\label{p:ar}
		\begin{cases}
			&F_{k}(x)  = 100(x_{k+1}-x_{k}^2), \quad\mathrm{mod}(k,4)=1, \notag\\
			&F_{k}(x) =1-4x_{k-1}, \quad\mathrm{mod}(k,4)=2, \notag\\
			&F_{k}(x) =1.25x_k-0.25x_k^3, \quad\mathrm{mod}(k,4)=3,\notag\\
			&F_{k}(x) = x_k, \quad\mathrm{mod}(k,4)=0,
			\quad 
			k=1,\dots, n. 
		\end{cases}  
	\end{equation}
	The initial guess $x_0=(x^1,\ldots,x^n)\in \mathbb{R}^n$ is set to $x^{k}=-1.2$ for $\mathrm{mod}(k,4)=1$, $x^{k}=1$ for  $\mathrm{mod}(k,4)=2$, $x^{k}=-1$ for  $\mathrm{mod}(k,4)=3$ and $x^{k}=20$ for  $\mathrm{mod}(k,4)=0$. 
\end{pro}

\begin{table}[h]
	\centering
	\begin{threeparttable}
		\caption{IT and CPU of different methods for the augmented Rosenbrock problem}
		\label{tab:ar}
		\setlength{\tabcolsep}{3.5pt} 
		\begin{tabular}{@{}llccccc@{}}
			\toprule
			\multirow{2}{*}{m} & \multirow{2}{*}{Metric} & \multicolumn{5}{c}{Algorithms} \\ 
			\cmidrule(lr){3-7}
			& & \multicolumn{1}{c}{ABNKAm}  & \multicolumn{1}{c}{ABNK2} & \multicolumn{1}{c}{RB-CNK} & \multicolumn{1}{c}{MRBNK} & \multicolumn{1}{c}{NGRKm}  \\ 
			\midrule
			\multirow{3}{*}{$10^3$} & IT & 24 & - & 
			- & 104 & 99046 \\ 
			& CPU & 0.0160  & - & - & 0.0881 & 15.729 \\ 
			&$\mathrm{SU}$ &   & - & - & 5.5 & 983.1\\ 
			\hline
			\multirow{3}{*}{$10^4$} & IT & 24 & - & 
			83 & 120 & -\\ 
			& CPU & 0.1502  & - & 0.5672 & 0.8724 & -  \\ 
			&$\mathrm{SU}$ &   &- & 3.8 & 5.8 & - \\ 
			\hline
			\multirow{3}{*}{$10^5$} & IT & 24 & - & 
			82 & 105 & -\\ 
			& CPU & 1.4908  & - & 5.7582 & 8.7826 & -\\ 
			&$\mathrm{SU}$ &   & - & 3.9 & 5.9 & -\\ 
			\hline
			\multirow{3}{*}{$10^6$} & IT & 24 & - & 
			83 & 102 & -\\ 
			& CPU & 14.382  & - & 60.528 & 87.384 & -\\ 
			&$\mathrm{SU}$ &   & - & 4.2 & 6.1 & -\\ 
			\bottomrule
		\end{tabular}
	\end{threeparttable}  
\end{table}

The numerical results are summarized in Table \ref{tab:ar}. As shown, the proposed ABNKAm method significantly outperforms other Kaczmarz-type methods in both iteration count and CPU time. Specifically, as the problem size increases from $10^3$ to $10^6$, ABNKAm consistently converges in only 24 iterations. This performance is substantially better than that of the fastest competing method, RB-CNK, which requires approximately 83 iterations at the $10^6$ scale, yielding a speed-up factor (SU) of 4.2 for ABNKAm.

The divergence of ABNK2 further underscores the crucial role of the momentum effect in ensuring convergence. Meanwhile, the row-based NGRKm method converges only at the smallest scale ($10^3$), and even then requires nearly $10^5$ iterations.

\begin{pro}\label{P5}
	The extended Powell badly scaled problem \cite{lukvsan1994inexact}:
	\begin{equation}
		\begin{cases}
			&F_{k}(x) =10000x_kx_{k+1}-1, \quad\mathrm{mod}(k,2)=1,\notag\\
			&F_{k}(x) = \exp(-x_{k-1})+\exp(-x_k)-1.0001, \quad\mathrm{mod}(k,2)=0,
			\quad 
			k=1,\dots, n. 
		\end{cases}  
	\end{equation}
	The initial value $x_0=(x^1,\ldots,x^n)\in \mathbb{R}^n$ is set to 
	$x^{k}=0$ for $\mathrm{mod}(k,2)=1$ and $x^{k}=1$ for  $\mathrm{mod}(k,2)=0$. 
\end{pro}

\begin{table}[htbp]
	\centering
	\begin{threeparttable}
		\caption{IT and CPU of different methods for the xtended Powell badly scaled problem}
		\label{tab:Extended Powell Badly Scaled Problem}
		\setlength{\tabcolsep}{3.5pt} 
		\begin{tabular}{@{}llccccc@{}}
			\toprule
			\multirow{2}{*}{m} & \multirow{2}{*}{Metric} & \multicolumn{5}{c}{Algorithms} \\ 
			\cmidrule(lr){3-7}
			& & \multicolumn{1}{c}{ABNKAm} & \multicolumn{1}{c}{ABNK2} & \multicolumn{1}{c}{RB-CNK} & \multicolumn{1}{c}{MRBNK} & \multicolumn{1}{c}{NGRKm}  \\ 
			\midrule
			\multirow{2}{*}{$1\cdot10^3$} & IT & 25 & - & - & - & - \\ 
			& CPU & 0.0357 & - & - & - & - \\ 
			
			\hline
			\multirow{2}{*}{$1\cdot10^4$} & IT & 28 & - & - & - & -  \\ 
			& CPU & 0.3369 & - & - & - & -  \\
			
			\hline
			\multirow{2}{*}{$1\cdot10^5$} & IT & 28 & - & - & - & - \\ 
			& CPU & 3.8527 & - & - & - & - \\ 
			
			\hline
			\multirow{2}{*}{$1\cdot10^6$} & IT & 28 & - & - & - & - \\ 
			& CPU & 35.097 & - & - & - & - \\ 
			\bottomrule
		\end{tabular}
	\end{threeparttable} 
\end{table}
\begin{figure}[htbp]
	\centering
	\begin{minipage}[t]{0.45\textwidth}
		\includegraphics[width=\textwidth, height=4.5cm]{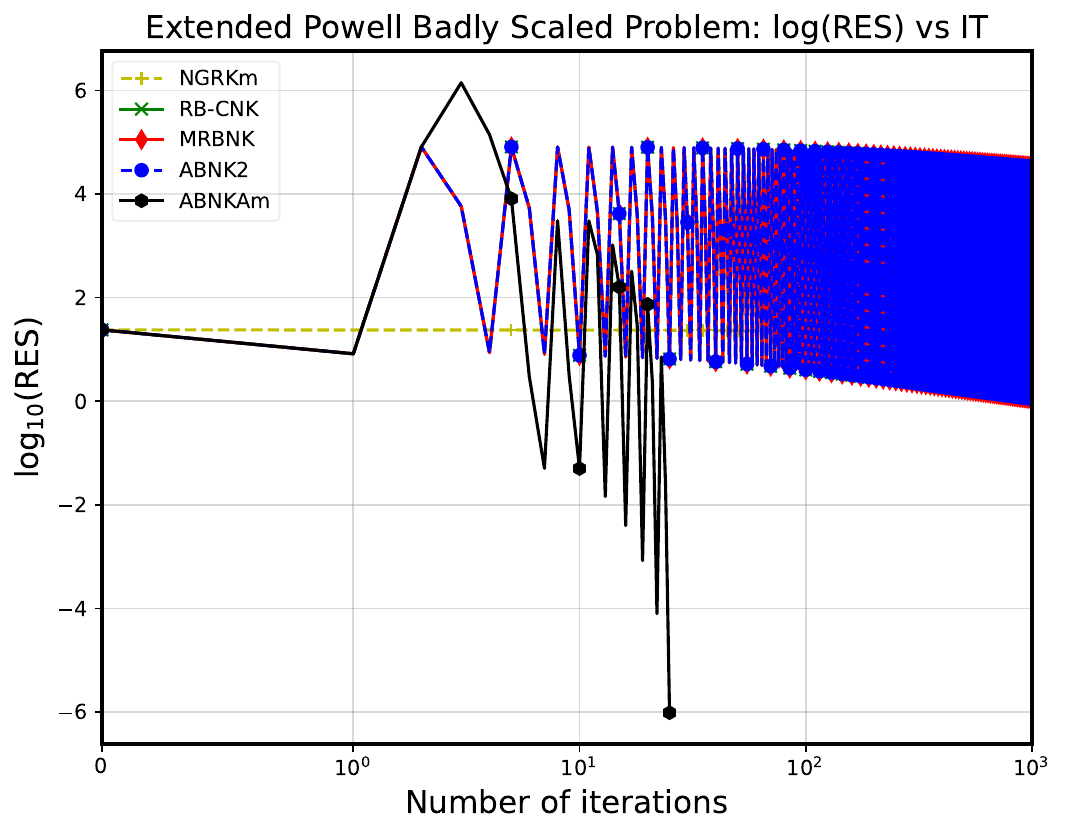}
	\end{minipage}
	\hfill
	\begin{minipage}[t]{0.45\textwidth}
		\includegraphics[width=\textwidth, height=4.5cm]{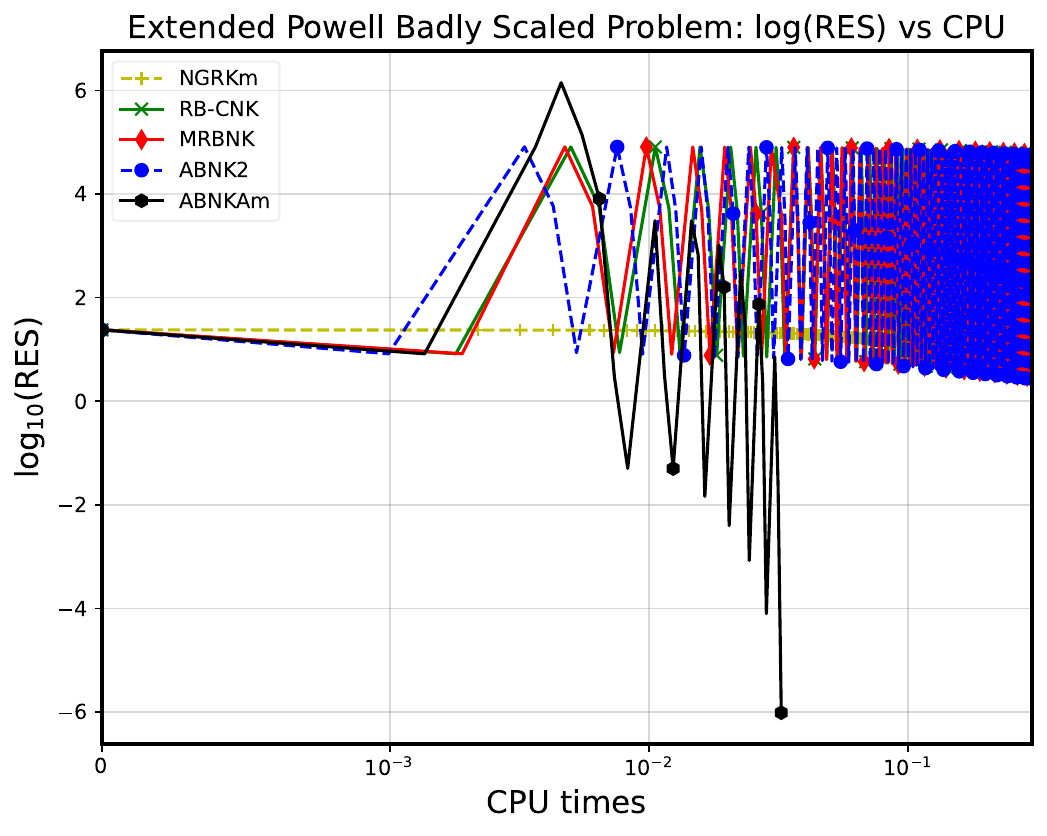}
	\end{minipage}
	\caption{Extended Powell badly scaled problem with $m=10^3$: 
		Curvers of RES against IT and CPU for different methods} 
	\label{fig:EPBS}
\end{figure}

\begin{figure}[htbp]
	\centering
	\begin{minipage}[t]{0.45\textwidth}
		\includegraphics[width=\textwidth, height=4.5cm]{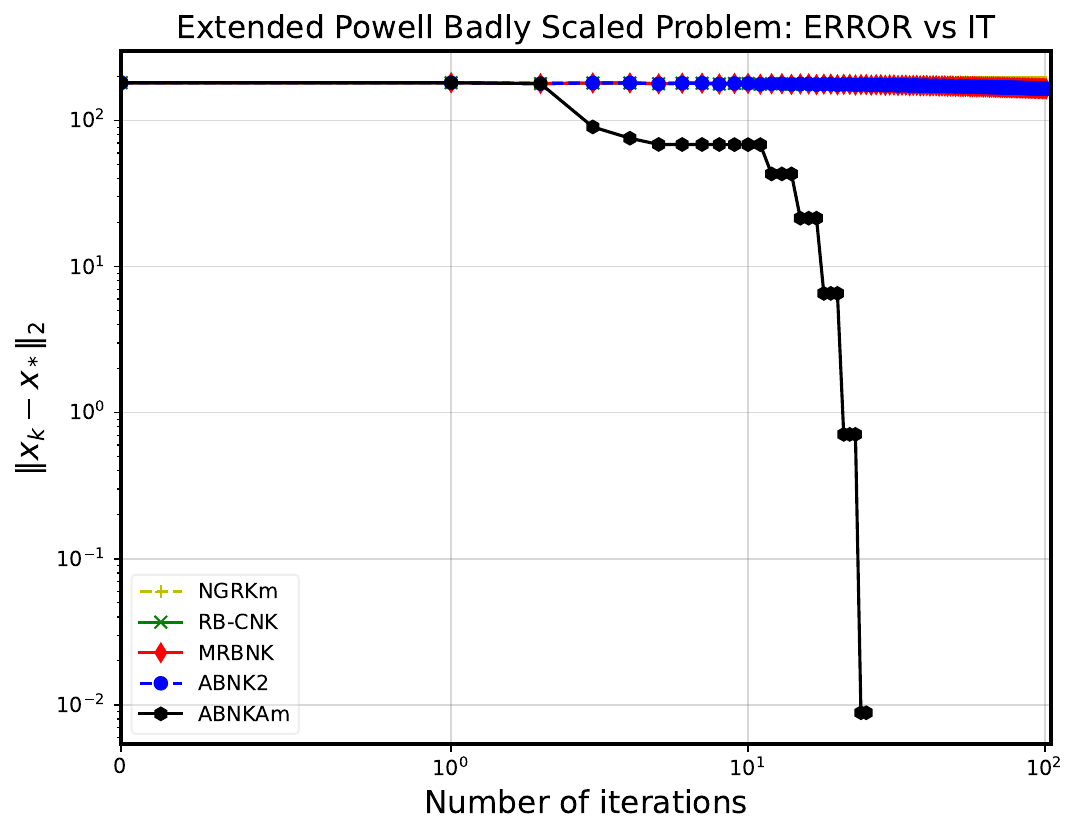}
	\end{minipage}
	\hfill
	\begin{minipage}[t]{0.45\textwidth}
		\includegraphics[width=\textwidth, height=4.5cm]{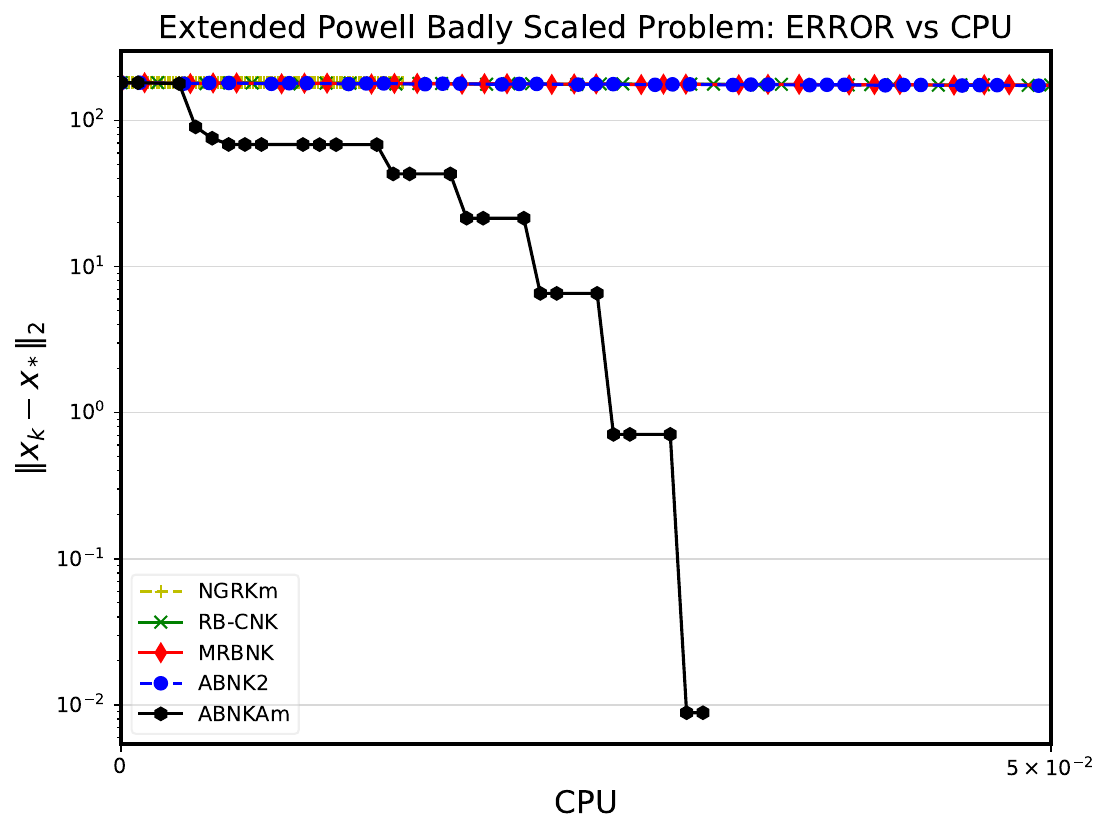}
	\end{minipage}
	\caption{Extended Powell badly scaled problem with $m=1\cdot 10^3$: 
		Curvers of errors against IT and CPU for different methods} 
	\label{fig:EPBS_ERROR}
\end{figure}

Numerical results for the extended Powell badly scaled problem are presented in Table \ref{tab:Extended Powell Badly Scaled Problem} and Figure \ref{fig:EPBS}. The proposed ABNKAm method achieves rapid convergence, requiring only 25-28 iterations across problem scales from $10^3$ to $10^6$. In contrast, all other Kaczmarz-type methods fail to converge, with their residual curves exhibiting persistent oscillations,  a manifestation of the ``seesaw" phenomenon. Although ABNKAm also shows oscillatory behavior, its residual curve maintains an overall descending trend and eventually converges, demonstrating the method's effectiveness for this class of nonlinear problems.

Figure \ref{fig:EPBS_ERROR} further illustrates the evolution of the iteration error $x_k - x_*$ with respect to both iteration count and CPU time. 
The sequence $\{x_k\}$ generated by ABNKAm converges rapidly toward the true solution, while other methods show negligible progress. It is worth noting that iterations with minimal error reduction in ABNKAm correspond to residual increases in Figure \ref{fig:EPBS}, indicating that such residual growth does not necessarily imply divergence. This behavior differs from classical monotonic descent algorithms: although some iterations may show limited immediate progress, they can facilitate subsequent convergence. These results suggest that ABNKAm partially mitigates the ``seesaw" effect, though developing more robust strategies remains an interesting direction for future work.

\subsection*{Further comments on the performance comparison between ABNKAm and other methods}
From the numerical results presented above, we have observed significant variation in performance 
among existing Kaczmarz-type methods across different problems. Specifically, row-action methods (NGRKm, MRNK, and NURK) are substantially less effective than block-oriented approaches (ABNK2, RB-CNK, MRBNK, and ABNKAm), with a notable decline in performance as problem size increases.

For sparse problems (i.e., those with sparse Jacobian matrices, excluding Problem \ref{P3}), the MRBNK method based on update formula \eqref{eq:mrbnk} achieves the best CPU time on Problems \ref{P1} and \ref{P2}. On Problem \ref{P1}, MRBNK requires approximately 90 iterations, compared to 126 for RB-CNK and 2913 for ABNK2. On Problem \ref{P2}, iteration counts 
range between 288 and 332 for MRBNK, between 399 and 470 for RB-CNK, and between 322 and 386 for ABNK2 
across different scales. In contrast, the RB-CNK method, which employs a different greedy criterion than MRBNK but uses the same update formula, shows an advantage on Problem \ref{P4}, converging in about 83 iterations versus 100-120 for MRBNK, while ABNK2 fails to converge completely.

Notably, although the ABNK2 method derived from \eqref{eq:abnk} avoids pseudoinverse computations, its per-iteration cost reduction is offset by increased iteration counts. Moreover, its performance degrades significantly on Problem \ref{P4}, and like all other existing Kaczmarz-type methods, it fails to solve Problem \ref{P5}.
For the dense Jacobian case (Problem \ref{P3}), ABNK2 leverages the averaging strategy to outperform pseudoinverse based methods RB-CNK and MRBNK in computational time. As the problem size grows from $10^3$ to $10^5$, ABNK2's runtime increases from  1 to over 2200 seconds, while RB-CNK and MRBNK increase from 6.3 to over 12000 and from 4.2 to nearly 8000 seconds, respectively.

In contrast to the performance variability exhibited by existing methods, the proposed ABNKAm method consistently outperforms all others in both iteration count and CPU time across all examples (Problems \ref{P1}-\ref{P5}). On sparse Problems \ref{P1} and \ref{P2}, it converges in approximately 9 and 180 iterations, respectively, surpassing MRBNK with speed-up factors of about 15 and 3. For Problem \ref{P4}, it requires only 24 iterations, outperforming RB-CNK by a factor of 4. Most notably, ABNKAm is the only method that successfully solves Problem \ref{P5}, converging within about 28 iterations. On the dense Problem \ref{P3}, it also surpasses ABNK2, achieving a nearly speed-up $2$. In summary, we have seen that ABNKAm converges stably, effectively and 
consistently for all problems in this subsection, while the convergence behavior of all other numerical methods  
is quite uncertain: each method may work well for some nonlinear examples, but may perform very poorly for other examples. 

\section{Concluding remarks}
\label{sec:con}
We have proposed a novel class of averaged block nonlinear Kaczmarz methods enhanced with momentum for solving nonlinear systems. An adaptive strategy for selecting step sizes and momentum parameters has been developed, relying solely on information from the iterative sequence. The proposed ABNKAm method has been validated through extensive numerical experiments, demonstrating not only high effectiveness but also remarkable competitiveness in both iteration count and computational time compared to existing nonlinear Kaczmarz-type methods.
Looking forward, the influence of greedy criteria on ABNKAm performance suggests a promising research direction. Developing specialized greedy criteria or efficient index subset selection strategies tailored to the ABNKAm framework could yield further computational gains, representing an interesting focus for our future work.




\bibliographystyle{siamplain}
\bibliography{references}
\end{document}